\documentclass[11pt]{article}%
\usepackage{amsmath}
\usepackage{amssymb}
\usepackage{amsthm}
\usepackage{amsfonts}
\usepackage[latin1]{inputenc}%
\usepackage{graphicx}
\providecommand{\U}[1]{\protect\rule{.1in}{.1in}}
\newtheorem{Theorem}{Theorem}[section]
\newtheorem{Proposition}[Theorem]{Proposition}
\newtheorem{Lemma}[Theorem]{Lemma}
\newtheorem{Corollary}[Theorem]{Corollary}
\theoremstyle{definition}
\newtheorem{Definition}[Theorem]{Definition}
\theoremstyle{remark}
\newtheorem{Example}[Theorem]{Example}
\newtheorem{Remark}[Theorem]{Remark}

\numberwithin{equation}{section}
\newcommand{\be}{\begin{equation}}
\newcommand{\ee}{\end{equation}}

\newenvironment{Proof}[1][Proof]{\noindent\textbf{#1:} }{\rule{0.5em}{0.5em}}
\begin{document}

\title{Generalized wave polynomials and transmutations related to perturbed Bessel equations }
\author{Vladislav V. Kravchenko, Sergii M. Torba and Jessica Yu. Santana-Bejarano\\{\small Departamento de Matemáticas, CINVESTAV del IPN, Unidad Querétaro, }\\{\small Libramiento Norponiente No. 2000, Fracc. Real de Juriquilla,
Querétaro, Qro. C.P. 76230 MEXICO}\\{\small e-mail: vkravchenko@math.cinvestav.edu.mx,
storba@math.cinvestav.edu.mx, {\ jsantana@math.cinvestav.mx} \thanks{The
authors acknowledge the support from CONACYT, Mexico via the projects 166141
and 222478.}}}
\maketitle

\begin{abstract}
The transmutation (transformation) operator associated with the perturbed
Bessel equation is considered. It is shown that its integral kernel can be
uniformly approximated by linear combinations of constructed here generalized
wave polynomials, solutions of a singular hyperbolic partial differential
equation arising in relation with the transmutation kernel. As a corollary of
this results an approximation of the regular solution of the perturbed Bessel
equation is proposed with corresponding estimates independent of the spectral parameter.
\end{abstract}

\section{Introduction}

The perturbed Bessel equation
\begin{equation}
-\mathbf{L}u(x):=-u^{\prime\prime}\left(  x\right)  +\left(  \frac{l\left(
l+1\right)  }{x^{2}}+q\left(  x\right)  \right)  u\left(  x\right)  =\lambda
u\left(  x\right)  ,\qquad l\geq-\frac{1}{2},\quad x\in\left(  0,b\right]
\label{Intro Bessel perturbed}%
\end{equation}
with $q$ being a continuous complex valued function and $\lambda$ an arbitrary
complex constant arises in numerous physical models describing wave
propagation in axially or spherically inhomogeneous media (see, e.g.,
\cite{Flugge} and \cite{Okamoto}). Solution of spectral problems for
(\ref{Intro Bessel perturbed}) is of special interest.

In \cite{Volk} and \cite{Stashevskaya} it was shown that there exists a
Volterra integral operator $\mathbf{T}$ defined by
\[
\mathbf{T}\left[  \varphi\right]  \left(  x\right)  =\varphi\left(  x\right)
+\int_{0}^{x}{K\left(  x,t\right)  \varphi\left(  t\right)  }dt
\]
for any $\varphi\in C\left(  \left[  0,b\right]  \right)  $, and with a
continuous kernel $K$ such that a regular solution of
(\ref{Intro Bessel perturbed}) can be written as
\[
\mathbf{u}(x,\lambda)=\mathbf{T}\left[  \mathbf{d}_{l}(x,\lambda)\right]
\]
where $\mathbf{d}_{l}(x,\lambda):=\sqrt{x}\mathbf{J}_{l+1/2}\left(
\sqrt{\lambda}x\right)  $ is a regular solution of the equation
\begin{equation}
-\mathbf{L}_{0}y(x):=-y^{\prime\prime}\left(  x\right)  +\frac{l\left(
l+1\right)  }{x^{2}}y\left(  x\right)  =\lambda y\left(  x\right)
\label{IntroBessel0}%
\end{equation}
with $\mathbf{J}_{p}$ standing for the Bessel function of the first kind and
of order $p$. Roughly speaking, the knowledge of the kernel $K$ allows one to
get rid of the potential $q$ reducing (\ref{Intro Bessel perturbed}) to a more
elementary (\ref{IntroBessel0}).

The transmutation kernel $K$ is a solution of a certain Goursat problem (see
Section 3 below) for the equation
\begin{equation}
\left(  \square-\frac{l\left(  l+1\right)  }{x^{2}}+\frac{l\left(  l+1\right)
}{t^{2}}-q\left(  x\right)  \right)  u\left(  x,t\right)  =0
\label{Intro KG-sing}%
\end{equation}
considered in the domain $0<t\leq x\leq b$. Here $\square:=\frac{\partial^{2}%
}{\partial x^{2}}-\frac{\partial^{2}}{\partial t^{2}}$. This partial
differential equation generalizes that arising in relation with the
transmutation kernel in the case of a regular one-dimensional Schrödinger
equation. In \cite{KT 2015 CAOT}, \cite{KT 2015 JCAM} it was shown that the
equation corresponding to the regular case
\[
\left(  \square-q\left(  x\right)  \right)  u\left(  x,t\right)  =0
\]
admits a complete system of solutions called generalized wave polynomials
which can be constructed using a certain recurrent integration procedure. In
particular, the transmutation kernel can be approximated globally by the
generalized wave polynomials, and the coefficients of approximation can be
found from the Goursat boundary conditions.

The aim of the present work is (a) to construct a system of generalized wave
polynomials $u_{k}$ corresponding to (\ref{Intro KG-sing}), (b) to prove that
the transmutation kernel $K$ admits a uniform approximation by their linear
combinations, (c) to show how the coefficients of approximation can be found,
and (d) to obtain a corresponding result concerning the approximate solution
$\mathbf{u}_{N}$ of (\ref{Intro Bessel perturbed}) obtained with the aid of
the approximate transmutation kernel establishing the smallness of the
difference between the exact and the approximate solutions.

The aim (a) is achieved in two steps. First, we consider an integral operator
$Y_{l,x}$ relating the operator $\frac{d^{2}}{dx^{2}}$ with the operator
$\frac{d^{2}}{dx^{2}}-\frac{l\left(  l+1\right)  }{x^{2}}$ on a class of
functions. This operator is known since long ago (see \cite{Sitnik} and
\cite{Sitnik Review} for historical remarks). We precise several known results
concerning its mapping properties and use its composition with $Y_{l,t}$ (same
operator but with respect to the variable $t$) in order to obtain a system of
solutions $U_{k}$ of the equation
\begin{equation}
\left(  \square-\frac{l\left(  l+1\right)  }{x^{2}}+\frac{l\left(  l+1\right)
}{t^{2}}\right)  U\left(  x,t\right)  =0 \label{Intro KG-sing0}%
\end{equation}
as images of the wave polynomials (a family of polynomial solutions of the
wave equation \cite{KhmeKravTorTrem}, \cite{KT 2015 CAOT}, \cite{KT 2015
JCAM}). Second, using a recent result from \cite{CasKravTor} regarding a
mapping property of the operator $\mathbf{T}$, we obtain the system of the
generalized wave polynomials $u_{k}$ for (\ref{Intro KG-sing}) as images under
the action of $\mathbf{T}$ of $U_{k}$ obtained on the previous step.

In order to achieve the aim (b) we study the preimage $k$ of the transmutation
kernel $K$, that is $k:=\mathbf{T}^{-1}\left[  K\right]  $ which satisfies a
Goursat problem for (\ref{Intro KG-sing0}). We prove that it admits a uniform
approximation by linear combinations of $U_{k}$ and deduce from there that $K$
admits a uniform approximation by linear combinations of $u_{k}$.

Regarding (c) we explain how the approximation coefficients can be found by
solving an approximation problem on the segment $\left[  0,b\right]  $ arising
from the Goursat boundary condition corresponding to the kernel $K$.

Finally, regarding (d) we prove that for the approximate solution
$\mathbf{u}_{N}$ of (\ref{Intro Bessel perturbed}) obtained via the
approximate transmutation kernel the estimate
\[
\left\vert \int_{0}^{x}{\left(  \mathbf{L}+\lambda\right)  \left[
\mathbf{u}_{N}\left(  s\right)  \right]  }ds\right\vert \leq2\varepsilon
\sqrt{x}%
\]
is valid, independent of $\lambda\in\mathbb{R}$ which makes the method
proposed in the present work attractive for solving
(\ref{Intro Bessel perturbed}) on large intervals with respect to $\lambda$. We illustrate the accuracy of the method based on the constructed approximate solutions by solving numerically two spectral problems.

\section{From regular to singular}

\subsection{Definition of the operator relating $\mathbf{L}_{0}$ with
$\frac{d^{2}}{dx^{2}}$}

Consider the following integral operator defined on $C\left(  \left[
0,b\right]  \right)  $,
\begin{equation}
Y_{l,x}f(x)=\frac{x^{-l}}{2^{l+\frac{1}{2}}\Gamma{\left(  l+\frac{3}%
{2}\right)  }}\int_{0}^{x}{\left(  x^{2}-s^{2}\right)  ^{l}f\left(  s\right)
}ds,\qquad l\geq-\frac{1}{2}. \label{Eq1.8}%
\end{equation}
With the change of the variable $s=xz$ (\ref{Eq1.8}) can be written in the
form
\begin{equation*}
Y_{l,x}f\left(  x\right)  =\frac{x^{l+1}}{2^{l+\frac{1}{2}}\Gamma{\left(
l+\frac{3}{2}\right)  }}\int_{0}^{1}{\left(  1-z^{2}\right)  ^{l}f\left(
xz\right)  }dz. %\label{CVY}%
\end{equation*}

\begin{Theorem}
\label{Teorema 1} For any $f\in C^{2}\left(  \left[  0,b\right]  \right)  $
the following equality is valid
\begin{equation}
\mathbf{L}_{0}Y_{l,x}{\left[  f\left(  x\right)  -f^{\prime}\left(  0\right)
x\right]  }=Y_{l,x}\left[  f^{\prime\prime}\left(  x\right)  \right]
,\qquad\forall x\in\left[  0,b\right]  . \label{Eq.T}%
\end{equation}

\end{Theorem}

\begin{Proof}
The right hand side of (\ref{Eq.T}) has the form
\[
Y_{l,x}\left[  f^{\prime\prime}\left(  x\right)  \right]  =\frac{x^{l+1}}%
{a}\int_{0}^{1}{\left(  1-z^{2}\right)  ^{l}f^{\prime\prime}\left(  xz\right)
}dz
\]
where $a=2^{l+\frac{1}{2}}\Gamma{\left(  l+\frac{3}{2}\right)  }$. For the
expression from the left hand side of (\ref{Eq.T}) we have
\[
\mathbf{L}_{0}Y_{l,x}\left[  f\left(  x\right)  -f^{\prime}\left(  0\right)
x\right]  =\frac{2\left(  l+1\right)  x^{l}}{a}\int_{0}^{1}{\left(
1-z^{2}\right)  ^{l}z\left[  f^{\prime}\left(  xz\right)  -f^{\prime}\left(
0\right)  \right]  }dz+\frac{x^{l+1}}{a}\int_{0}^{1}{\left(  1-z^{2}\right)
^{l}z^{2}f^{\prime\prime}\left(  xz\right)  }dz.
\]
Therefore,
\begin{align*}
\mathbf{L}_{0}Y_{l,x}\left[  f\left(  x\right)  -f^{\prime}\left(  0\right)
x\right]   & -Y_{l,x}\left[  f^{\prime\prime}\left(  x\right)  \right] \\
&  =\frac{2\left(  l+1\right)  x^{l}}{a}\int_{0}^{1}{\left(  1-z^{2}\right)
^{l}z\left[  f^{\prime}\left(  xz\right)  -f^{\prime}\left(  0\right)
\right]  }dz-\frac{x^{l+1}}{a}\int_{0}^{1}{\left(  1-z^{2}\right)
^{l+1}f^{\prime\prime}\left(  xz\right)  }dz\\
&  =-\frac{x^{l}}{a}f^{\prime}\left(  0\right)  \left[  2\left(  l+1\right)
\int_{0}^{1}{\left(  1-z^{2}\right)  ^{l}z}dz-1\right]  =0,
\end{align*}
where an
%$u=(1-z^{2})^{l+1},$ $dv=f''\left(xz\right)dz$ for
integration by parts was used as well as formula 3.251.1 from \cite{GradsRyz}.
\end{Proof}

\subsection{Properties of the operator $Y_{l,x}$}

\begin{Proposition}
\label{P1.1}
\begin{equation*}
Y_{l,x}\left[  x^{k}\right]  =C_{k}x^{k+l+1},\qquad k=0,1,2,\ldots%\label{PX}%
\end{equation*}
with
\begin{equation*}
C_{k}:=\frac{\Gamma{\left(  \frac{k+1}{2}\right)  }\Gamma{\left(  l+1\right)
}}{2^{l+\frac{3}{2}}\Gamma{\left(  l+\frac{3}{2}\right)  }\Gamma{\left(
l+\frac{k+3}{2}\right)  }}. %\label{Ck}%
\end{equation*}
\end{Proposition}

For the proof one can use the change of the variable $s=x\sin\varphi$ and
formula 3.642.1 from \cite{GradsRyz}.

\begin{Proposition}
For any $f\in C\left(  \left[  0,b\right]  \right)  $ the following equalities
are valid
\begin{align*}
&  \lim_{x\rightarrow0}{Y_{l,x}f\left(  x\right)  }=0,\qquad\forall
l\geq-\frac{1}{2},\\
&  \lim_{x\rightarrow0}{\left(  Y_{l,x}f\right)  ^{\prime}\left(  x\right)
}=0,\qquad\forall l>0,\\
&  \lim_{x\rightarrow0}{\left(  Y_{0,x}f\right)  ^{\prime}\left(  x\right)
}=\sqrt{\frac{2}{\pi}}f\left(  0\right)  .
\end{align*}

\end{Proposition}

The proof is elementary and is based on the mean value theorem, L'Hospital's
rule and formula 3.251.1 from \cite{GradsRyz}.

The operator ${Y_{l,x}}$ is related to the Bessel equation
\begin{equation*}
-\mathbf{L}_{0}u=\lambda u %\label{Eq1.12}%
\end{equation*}
where $\lambda$ is an arbitrary complex number. The equation possesses a
regular solution satisfying the asymptotic relations $u_{l}\sim x^{l+1}$ and
$u_{l}^{\prime}\sim\left(  l+1\right)  x^{l}$, $x\rightarrow0$ and given by
the formula
\[
u_{l}\left(  x,\lambda\right)  =\Gamma{\left(  l+\frac{3}{2}\right)
}2^{l+\frac{1}{2}}\lambda^{-\frac{2l+1}{4}}\sqrt{x}\mathbf{J}_{l+\frac{1}{2}%
}\left(  \sqrt{\lambda}x\right)  ,
\]
where $\mathbf{J}_{l+\frac{1}{2}}$ is the Bessel function of the first kind
and of order $l+\frac{1}{2}$, $l\geq-\frac{1}{2}$. The solution $u_{l}\left(
x,\lambda\right)  $ up to a multiplicative constant is the image of the
solution $\cos{\sqrt{\lambda}x}$ of the equation $v^{\prime\prime}+\lambda
v=0$ under the action of the operator $Y_{l,x}$. Indeed,
\[%
\begin{split}
Y_{l,x}\left[  \cos{\sqrt{\lambda}x}\right]   & = \frac{\sqrt{\pi x}%
\Gamma(l+1)}{\lambda^{\frac{2l+1}4}\Gamma(l+3/2)}\frac{\left( \frac
{\sqrt{\lambda} x}2\right) ^{l+1/2}}{\sqrt{\pi}\Gamma(l+1)}\int_{0}^{1}
(1-z^{2})^{l} \cos(\sqrt{\lambda} xz)\,dz\\
& = \frac{\sqrt{\pi}\Gamma{\left(  l+1\right)  }}{2\Gamma{\left(  l+\frac
{3}{2}\right)  }}\sqrt{x}\lambda^{-\frac{2l+1}{4}}\mathbf{J}_{l+\frac{1}{2}
}\left(  \sqrt{\lambda}x\right) ,
\end{split}
\]
where formula \cite[(5.10.2)]{Lebedev} was used.

Thus,
\[
\frac{2^{l+\frac{3}{2}}\Gamma^{2}{\left(  l+\frac{3}{2}\right)  }}{\sqrt{\pi
}\Gamma{\left(  l+1\right)  }}Y_{l,x}\left[  \cos{\sqrt{\lambda}x}\right]
=u_{l}\left(  x,\lambda\right)  .
\]
For the second linearly independent solution of $v^{\prime\prime}+\lambda v=0
$ we obtain the relation
\[
Y_{l,x}\left[  \sin{\left(  \sqrt{\lambda}x\right)  }\right]  =\frac{\sqrt
{\pi}\Gamma{\left(  l+1\right)  }}{2\Gamma{\left(  l+\frac{3}{2}\right)  }%
}\lambda^{-\frac{2l+1}{4}}\sqrt{x}H_{l+\frac{1}{2}}\left(  \sqrt{\lambda
}x\right)  ,
\]
where $H_{l+\frac{1}{2}}$ stands for the Struve function of order $l+\frac
{1}{2}$. Using Theorem \ref{Teorema 1} it is easy to verify that
$h(x,\lambda):=Y_{l,x}\left[  \sin{\left(  \sqrt{\lambda}x\right)  }\right]  $
is a solution of the inhomogeneous Bessel equation
\[
\left(  \mathbf{L}_{0}+\lambda\right)  h(x,\lambda)=2\left(  l+1\right)
\sqrt{\lambda}C_{1}x^{l}.
\]

\begin{Proposition}
The operators $Y_{l,x}:C\left(  \left[  0,b\right]  \right)  \rightarrow
C\left(  \left[  0,b\right]  \right)  $ and $\frac{d}{dx}Y_{l,x}:C^{1}\left(
\left[  0,b\right]  \right)  \rightarrow C\left(  \left[  0,b\right]  \right)
$ are continuous.
\end{Proposition}

The proof is straightforward.

\subsection{The inverse operator}

To find the inverse operator we represent our operator as an Erdélyi-Kober
operator, which are modifications of Riemann-Liouville fractional integrals
and derivatives \cite[§18]{SamKilMar}, and use their well developed theory.
Consider the following Erdélyi-Kober operator
\begin{align}
I_{a;\sigma,\eta}^{\alpha}{f\left(  x\right)  }  &  =\frac{\sigma
x^{-\sigma\left(  \alpha+\eta\right)  }}{\Gamma{\left(  \alpha\right)  }}%
\int_{a}^{x}{\left(  x^{\sigma}-t^{\sigma}\right)  ^{\alpha-1}t^{\sigma
\eta+\sigma-1}f\left(  t\right)  }dt,\qquad\alpha>0,\label{Eqtd}\\
I_{a;\sigma,\eta}^{\alpha}{f\left(  x\right)  }  &  =x^{-\sigma\left(
\alpha+\eta\right)  }\left(  \frac{d}{\sigma x^{\sigma-1}dx}\right)
^{n}x^{\sigma\left(  \alpha+n+\eta\right)  }I_{a;\sigma,\eta}^{\alpha
+n}f\left(  x\right)  ,\qquad\alpha>-n\label{Eqtd1}%
\end{align}
where $0\leq a<x<\infty$ for any $\sigma$ real or $-\infty\leq a<x<\infty$ for
$\sigma$ integer. Note that the value on $n$ in \eqref{Eqtd1} can be arbitrary
satisfying $n>-\alpha$.

The expression for the inverse operator is given by \cite[§18]{SamKilMar}
\begin{equation}
\left(  I_{a;\sigma,\eta}^{\alpha}\right)  ^{-1}=I_{a;\sigma,\eta+\alpha
}^{-\alpha} \label{Eq1.34}%
\end{equation}
%\left(  I^{\alpha}_{b_{-},\eta}\right)  ^{-1}f\left(  x\right)   &
%=I^{-\alpha}_{b_{-},\eta+\alpha}f\left(  x\right)  .

\begin{Remark}
The operator $Y_{l,x}$ can be represented in terms of an Erdélyi-Kober
operator as follows
\begin{equation*}
Y_{l,x}{f\left(  x\right)  }=\frac{\Gamma{\left(  l+1\right)  }x^{l+1}%
}{2^{l+\frac{3}{2}}\Gamma{\left(  l+\frac{3}{2}\right)  }}I_{0;2,-\frac{1}{2}%
}^{l+1}{f\left(  x\right)  .} %\label{OEK}%
\end{equation*}
\end{Remark}
From this and from (\ref{Eq1.34}) we obtain the inverse of $Y_{l,x}$.

\begin{Proposition}
Let $n$ be arbitrary integer satisfying $n>l+1$. Then the inverse operator of
$Y_{l,x}$ can be taken in the form
\begin{equation*}
Y_{l,x}^{-1}{f\left(  x\right)  }=\frac{2^{l+\frac{5}{2}}\Gamma{\left(
l+\frac{3}{2}\right)  }x}{\Gamma{\left(  l+1\right)  }\Gamma\left(  {n-\left(
l+1\right)  }\right)  }\left(  \frac{d}{2xdx}\right)  ^{n}\int_{0}^{x}{\left(
x^{2}-s^{2}\right)  ^{-\left(  l+2\right)  +n}s^{l+1}f\left(  s\right)  }ds.
%\label{IEK}%
\end{equation*}
\end{Proposition}

\section{Transmutation for perturbed Bessel equations}

\subsection{Existence of the transmutation}

Consider the perturbed Bessel equation
\begin{equation}
-u^{\prime\prime}\left(  x\right)  +\left(  \frac{l\left(  l+1\right)  }%
{x^{2}}+q\left(  x\right)  \right)  u\left(  x\right)  =\lambda u\left(
x\right)  ,\qquad l\geq-\frac{1}{2},\quad x\in\left(  0,b\right]
\label{Bessel}%
\end{equation}
where $q$ is a complex valued continuous function on $\left[  0,b\right]  $
and $\lambda$ is an arbitrary complex number.

When $q\equiv0$ the function
\begin{equation}
\mathbf{d}_{l}\left(  x,\lambda\right)  :=\sqrt{x}\mathbf{J}_{l+\frac{1}{2}%
}\left(  \sqrt{\lambda}x\right)  =\sqrt{x}\sum_{n=0}^{\infty}{\frac{\left(
-1\right)  ^{n}}{n!\Gamma{\left(  n+l+\frac{3}{2}\right)  }}\left(
\frac{\sqrt{\lambda}x}{2}\right)  ^{2n+l+\frac{1}{2}}} \label{dl}%
\end{equation}
is a regular solution of (\ref{Bessel}).

In \cite{Volk} it was proved that a regular solution of (\ref{Bessel}) is the
image of $\mathbf{d}_{l}\left(  x,\lambda\right)  $ under the action of the
integral operator defined by
\begin{equation}
\label{T}\mathbf{T}\left[  \varphi\right]  \left(  x\right)  =\varphi\left(
x\right)  +\int^{x}_{0}{K\left(  x,t\right)  \varphi\left(  t\right)  }dt
\end{equation}
for any $\varphi\in C\left(  \left[  0,b\right]  \right)  $, and with the
kernel $K$ described by the following theorem

\begin{Theorem}
[\cite{Volk}]\label{Th Kernel K} There exists and only one function $K\left(
x,t\right)  $ continuous in the triangle $0\leq t\leq x\leq b$ such that the
following properties are satisfied

\begin{enumerate}
\item $K$ is a solution of the equation
\begin{equation}
\left(  \square-\frac{l\left(  l+1\right)  }{x^{2}}+\frac{l\left(  l+1\right)
}{t^{2}}-q\left(  x\right)  \right)  u\left(  x,t\right)  =0 \label{waveq}%
\end{equation}
in the domain $0<t\leq x\leq b$.

\item $\frac{dK\left(  x,x\right)  }{dx}=\frac{q\left(  x\right)  }{2},\qquad
x\in\left[  0,b\right]  $.

\item
\begin{equation}
K\left(  x,0\right)  =0,\qquad x\in\left[  0,b\right]  \label{K(x,0)=0}%
\end{equation}
and additionally
\begin{equation}
\lim_{t\rightarrow0}{K\left(  x,t\right)  \cdot t^{l}}=0,\qquad\text{when
}-\frac{1}{2}\leq l<0. \label{limK}%
\end{equation}

\end{enumerate}
\end{Theorem}

\begin{Remark}
The property 3. is slightly modified as compared to the original formulation
of this theorem in \cite{Volk}, where only (\ref{limK}) appears. However the
fact that $K\left(  x,0\right)  =0$ for any $l\geq-\frac{1}{2}$ can be
established easily using the results from \cite{Volk}. Indeed, according to
\cite{Volk} (with the corrections similar to those presented in the Appendix
\ref{Appendix}) the function $K$ admits the representation $K\left(
x,t\right)  =\left(  z-s\right)  ^{-l}v\left(  z,s\right)  $ where $z=\frac
{1}{4}\left(  x+t\right)  ^{2},$ $s=\frac{1}{4}\left(  x-t\right)  ^{2}$ and
$v\left(  z,s\right)  =O\left(  \left(  z-s\right)  ^{1/2+l-\varepsilon
}\right)  ,$ $0<\varepsilon<1/2$. From where we obtain (\ref{K(x,0)=0}) by
letting $t\rightarrow0$.
\end{Remark}

The main result of \cite{Volk} establishes that the function
\[
u_{l}\left(  x,\lambda\right)  =\mathbf{T}\left[  \mathbf{d}_{l}\left(
x,\lambda\right)  \right]
\]
is a regular solution of (\ref{Bessel}) with $\mathbf{T}$ having the form
(\ref{T}) with $K$ being defined by Theorem \ref{Th Kernel K}.

Note also that the operator $\mathbf{T}:C([0,b])\to C([0,b])$ is a bijection.

\subsection{The transmutation property}

Denote as $C_{0}\left(  \left[  0,b\right]  \right)  $, $C_{0}^{1}\left(
\left[  0,b\right]  \right)  $ and {$\mathcal{A}_{l}$} the following
functional spaces
\[
C_{0}\left(  \left[  0,b\right]  \right)  :=\left\{  \varphi\in C\left(
\left[  0,b\right]  \right)  :\varphi\left(  0\right)  =0\right\}  ,\qquad
C_{0}^{1}\left(  \left[  0,b\right]  \right)  :=\left\{  \varphi\in
C^{1}\left(  \left[  0,b\right]  \right)  :\varphi\left(  0\right)
=\varphi^{\prime}\left(  0\right)  =0\right\}
\]
and
\[
\mathcal{A}_{l}=C_{0}\left(  \left[  0,b\right]  \right)  \cap C^{2}\left(
\left(  0,b\right]  \right)  \cap\left\{  \varphi:\varphi^{\prime}\in C\left[
0,b\right]  ,\text{ if }l>0,\text{ and }\varphi^{\prime}\left(  t\right)
=O\left(  t^{l}\right) \text{ when }t\rightarrow0,\text{ if }l<0\right\} .
\]

\begin{Theorem}
\label{Th Transmutation property}For any $\varphi\in\mathcal{A}_{l}$ the
following equality is valid
\begin{equation*}
\mathbf{L}\mathbf{T}\varphi=\mathbf{TL}_{0}\varphi. %\label{3.4}%
\end{equation*}
\end{Theorem}

\begin{Proof}
Developing the left-hand side we obtain {%
\begin{align*}
\mathbf{L}\mathbf{T}\left[  \varphi\right]  \left(  x\right)   &
=\varphi^{\prime\prime}\left(  x\right)  +\int_{0}^{x}{\frac{\partial
^{2}K\left(  x,t\right)  }{\partial x^{2}}\varphi\left(  t\right)  }%
dt+\frac{\partial K\left(  x,t\right)  }{\partial x}\Big|_{t=x}\varphi\left(
x\right)  +\frac{dK\left(  x,x\right)  }{dx}\varphi\left(  x\right)  +K\left(
x,x\right)  \varphi^{\prime}\left(  x\right) \\
&  \quad-\frac{l\left(  l+1\right)  }{x^{2}}\varphi\left(  x\right)
-\frac{l\left(  l+1\right)  }{x^{2}}\int_{0}^{x}{K\left(  x,t\right)
\varphi\left(  t\right)  }dt-q\left(  x\right)  \varphi\left(  x\right)
-q\left(  x\right)  \int_{0}^{x}{K\left(  x,t\right)  \varphi\left(  t\right)
}dt.
\end{align*}
} On the other hand,
\[
\mathbf{TL}_{0}\left[  \varphi\right]  \left(  x\right)  =\varphi
^{\prime\prime}\left(  x\right)  +\int_{0}^{x}{K\left(  x,t\right)
\varphi^{\prime\prime}\left(  t\right)  }dt-\frac{l\left(  l+1\right)  }%
{x^{2}}\varphi\left(  x\right)  -\int_{0}^{x}{K\left(  x,t\right)
\frac{l\left(  l+1\right)  }{t^{2}}\varphi\left(  t\right)  }dt.
\]
Integration by parts applied to the term $\int_{0}^{x}{K\left(  x,t\right)
\varphi^{\prime\prime}\left(  t\right)  }dt$ gives
\begin{align*}
\int_{0}^{x}{K\left(  x,t\right)  \varphi^{\prime\prime}\left(  t\right)  }dt
&  =K\left(  x,x\right)  \varphi^{\prime}\left(  x\right)  -K\left(
x,0\right)  \varphi^{\prime}\left(  0\right)  -\frac{\partial K\left(
x,t\right)  }{\partial t}\Big|_{t=x}\varphi\left(  x\right)  +\frac{\partial
K\left(  x,t\right)  }{\partial t}\Big|_{t=0}\varphi\left(  0\right) \\
&  \quad+\int_{0}^{x}{\frac{\partial^{2}K\left(  x,t\right)  }{\partial t^{2}%
}\varphi\left(  t\right)  }dt.
\end{align*}
Thus,%
\[
\int_{0}^{x}{K\left(  x,t\right)  \varphi^{\prime\prime}\left(  t\right)
}dt=K\left(  x,x\right)  \varphi^{\prime}\left(  x\right)  -\frac{\partial
K\left(  x,t\right)  }{\partial t}\Big|_{t=x}\varphi\left(  x\right)
+\int_{0}^{x}{\frac{\partial^{2}K\left(  x,t\right)  }{\partial t^{2}}%
\varphi\left(  t\right)  }dt
\]
where we used the supposition $\varphi\in{\mathcal{A}}_{l}$ as well as
property 3. from Theorem \ref{Th Kernel K}. Now we have
\begin{multline*}
\mathbf{L}\mathbf{T}\left[  \varphi\right]  \left(  x\right)  -\mathbf{TL}%
_{0}\left[  \varphi\right]  \left(  x\right) \\
=\int_{0}^{x}{\left(  \frac{\partial^{2}K\left(  x,t\right)  }{\partial x^{2}%
}-\frac{\partial^{2}K\left(  x,t\right)  }{\partial t^{2}}-\frac{l\left(
l+1\right)  }{x^{2}}K\left(  x,t\right)  +\frac{l\left(  l+1\right)  }{t^{2}%
}K\left(  x,t\right)  -q\left(  x\right)  K\left(  x,t\right)  \right)
\varphi\left(  t\right)  }dt\\
+2\frac{dK\left(  x,x\right)  }{dx}\varphi\left(  x\right)  -q\left(
x\right)  \varphi\left(  x\right)  =0,
\end{multline*}
due to properties 1. and 2. from Theorem \ref{Th Kernel K} satisfied by $K$.
\end{Proof}

Construction of the kernel $K$ in a closed form is a difficult task.
Nevertheless one can always calculate the image of the power $x^{2k+l+1}$ for
any $k\in0,1,2,\ldots$. These are precisely the powers of $x$ which arise in
the power series representation (\ref{dl}) of the function $\mathbf{d}_{l}$.
To formulate this mapping property of the operator $\mathbf{T}$ let us define
the following system of recursive integrals \cite{CasKravTor}
\begin{equation}%
\begin{split}
\widetilde{X}^{(0)}  &  \equiv1,\qquad\widetilde{X}^{(-1)}\equiv0,\\
\widetilde{X}^{(n)}(x)  &  =%
\begin{cases}
\displaystyle\int_{0}^{x}u_{0}^{2}(t)\widetilde{X}^{(n-1)}(t)dt, & \text{if
}n\text{ is odd},\\
-\displaystyle\int_{0}^{x}\frac{\widetilde{X}^{(n-1)}(t)}{u_{0}^{2}(t)}\,dt, &
\text{if }n\text{ is even}.
\end{cases}
\end{split}
\label{Xtilde}%
\end{equation}
Here $u_{0}$ denotes a solution of the equation
\[
-u^{\prime\prime}\left(  x\right)  +\left(  \frac{l\left(  l+1\right)  }%
{x^{2}}+q\left(  x\right)  \right)  u\left(  x\right)  =0
\]
satisfying the asymptotic relation $u(x)\sim x^{l+1}$ when $x\rightarrow0$.

Functions (\ref{Xtilde}) arise in relation with the spectral parameter power
series (SPPS) representation of the regular solutions of
(\ref{Intro Bessel perturbed}). We keep the notation $\widetilde{X}$ for
consistency with the notations from other publications on the SPPS method,
see, e.g., \cite{KrCV08}, \cite{APFT}, \cite{KrPorter2010}, \cite{KKRosu}. The
functions $\widetilde{X}^{(n)}$ are well defined if one assumes that $u_{0}$
has no zero on $(0,b]$. For example, this is true when $q(x)\geq0$ for all
$x\in(0,b]$ (see \cite[Corollary 3.3]{CasKravTor}). However, in fact we need
to consider not the functions $\widetilde{X}^{(n)}$ themselves but their
products with $u_{0}$ (see equality (\ref{MappingProperty}) below) which can
be well defined even when $u_{0}$ has zeros on $(0,b]$. In order not to
overload this paper with corresponding technical details we assume that
$u_{0}$ has no zero on $(0,b]$ but emphasize that the results remain valid
without this restriction.

\begin{Proposition}
[{\cite[Corollary 3.3]{CasKravTor}}]\label{Prop Mapping property}
\begin{equation}
\mathbf{T}[x^{2k+l+1}]=(-1)^{k}2^{2k}k!\left(  l+\frac{3}{2}\right)  _{k}%
u_{0}(x)\widetilde{X}^{(2k)}(x). \label{MappingProperty}%
\end{equation}

\end{Proposition}

\subsection{Extension of the transmutation kernel}

In this subsection we show that the kernel $K\left(  x,t\right) $, defined up
to now on the triangle $\left\{  0\leq t\leq x\leq b\right\} $, admits a
natural extension onto the square {$\left\{  0\leq t\leq x\leq b\right\}
\cup\left\{  0\leq x\leq t\leq b\right\}  $}. For this let us consider the
inverse operator ${\mathbf{T}}^{-1}$ which has the form
\[
{\mathbf{T}}^{-1}\left[  \varphi\right]  \left(  x\right)  =\varphi\left(
x\right)  -\int_{0}^{x}{L\left(  x,t\right)  \varphi\left(  t\right)  }dt.
\]
Using similar reasoning as with respect to the kernel $K$ it is easy to show
that the kernel $L$ is a solution of the equation
\begin{equation*}
\frac{\partial^{2}L\left(  x,t\right)  }{\partial x^{2}}-\frac{\partial
^{2}L\left(  x,t\right)  }{\partial t^{2}}-\frac{l\left(  l+1\right)  }{x^{2}
}L\left(  x,t\right)  +\frac{l\left(  l+1\right)  }{t^{2}}L\left(  x,t\right)
+q\left(  t\right)  L\left(  x,t\right)  =0 %\label{EV}%
\end{equation*}
and satisfies the following boundary conditions
\[
\frac{dL\left(  x,x\right)  }{dx}=\frac{1}{2}q\left(  x\right)  ,\qquad
L\left(  x,0\right)  =0,\qquad\text{and additionally }\lim_{t\rightarrow
0}{L\left(  x,t\right)  t^{l}}=0\text{ when }-\frac{1}{2}\leq l<0.
\]
Now the extension of the kernel $K(x,t)$ onto the triangle $0\leq x\leq t\leq
b$ is defined as follows $K(x,t):=L(t,x)$. Hence the extended kernel $K$
satisfies (\ref{waveq}) on $\left\{  0<t<x\leq b\right\}  \cup\left\{
0<x<t\leq b\right\}  $. Additionally to conditions \eqref{K(x,0)=0} and
\eqref{limK} it satisfies $K\left(  0,t\right)  =0$ and, for $-\frac{1}{2}\leq
l<0$, the condition $\lim_{x\rightarrow0}{K\left(  x,t\right)  \cdot x^{l}
=0}$. Note that the extended kernel $K$ is continuous on the whole square
$[0,b]\times[0,b]$.

From now on we consider this extended kernel and use the same notation $K$ for it.

\subsection{The preimage of the transmutation kernel}

The kernel $K{\left(  x,t\right)  }$ itself is the image of a certain function
$k\left(  x,t\right)  $ defined on the square $[0,b]\times[0,b]$ under the
action of the operator $\mathbf{T}$,
\begin{equation}
K\left(  x,\tau\right)  =\mathbf{T}\left[  k\left(  x,\tau\right)  \right]
=k\left(  x,\tau\right)  +\int_{0}^{x}{K\left(  x,t\right)  k\left(
t,\tau\right)  }dt. \label{K}%
\end{equation}
Note that we are applying the operator $\mathbf{T}$ with respect to the
variable $x$, the variable $\tau$ serves as a parameter.

The extension of the transmutation kernel $K$ introduced in the previous
section allows us to write down the preimage $k$ in the form \
\begin{equation}
k\left(  x,\tau\right)  :=\mathbf{T}^{-1}\left[  K\left(  x,\tau\right)
\right]  =K\left(  x,\tau\right)  -\int_{0}^{x}{K\left(  t,x\right)  K\left(
t,\tau\right)  }dt. \label{k}%
\end{equation}
Now we are interested in establishing certain basic properties of $k$ in order
to prove that it can be approximated by wave polynomials on a suitable domain.
For this we need to show that $k$ is a solution of equation
(\ref{Intro KG-sing0}). Considering the triangle $0<t<x<b$ as the domain does
not lead to the aim. Indeed, the kernel $K$ is not necessary a solution of
(\ref{waveq}) on the whole square {$\left\{  0<t\leq x<b\right\}  \cup\left\{
0<x\leq t<b\right\}  $ but only on the union of open triangles $\left\{
0<t<x<b\right\}  \cup\left\{  0<x<t<b\right\} $. Hence one can not verify that
$k$ is a solution of \eqref{Intro KG-sing0} by direct substitution of
(\ref{k}) into equation (\ref{Intro KG-sing0}), this would require in the
calculation that $K$ is the solution of (\ref{waveq}) on the whole square.
However, consideration of triangle $\left\{  0<x<t<b\right\}  $ as the domain
for $k$ does not provoke any difficulty. For that reason from now on we will
consider the triangle $\left\{  0<x<t<b\right\}  $ as the domain for all the
constructions. }

We begin with the following auxiliary result in this direction.

\begin{Lemma}
\label{Lemma transmutation of k} For any $0<x<\tau<b$ the following equality
holds
\begin{equation}
\mathbf{L}\mathbf{T}\left[  k\left(  x,\tau\right)  \right]  =\mathbf{TL}%
_{0}\left[  k\left(  x,\tau\right)  \right] . \label{Tk}%
\end{equation}

\end{Lemma}

\begin{Proof}
On the left-hand side of (\ref{Tk}) we have
\[%
\begin{split}
\mathbf{L}\mathbf{T}\left[  k\left(  x,\tau\right)  \right]   & ={\frac
{\partial^{2}k\left(  x,\tau\right)  }{\partial x^{2}}}+\int_{0}^{x}%
{\frac{\partial^{2}K\left(  x,t\right)  }{\partial x^{2}}k\left(
t,\tau\right)  }dt+\frac{\partial K\left(  x,t\right)  }{\partial
x}\Big|_{t=x}k\left(  x,\tau\right)  +\frac{dK\left(  x,x\right)  }%
{dx}k\left(  x,\tau\right) \\
&  \quad+K\left(  x,x\right)  {\frac{\partial k\left(  x,\tau\right)
}{\partial x}} -\frac{l\left(  l+1\right)  }{x^{2}}k\left(  x,\tau\right)
-\frac{l\left(  l+1\right)  }{x^{2}}\int_{0}^{x}{K\left(  x,t\right)  k\left(
t,\tau\right)  }dt\\
& \quad-q\left(  x\right)  k\left(  x,\tau\right)  -q\left(  x\right)
\int_{0}^{x}{K\left(  x,t\right)  k\left(  t,\tau\right)  }dt.
\end{split}
\]
The right-hand side gives
\[
\mathbf{TL}_{0}\left[  k\left(  x,\tau\right)  \right]  ={\frac{\partial
^{2}k\left(  x,\tau\right)  }{\partial x^{2}}}+\int_{0}^{x}{K\left(
x,t\right)  \frac{\partial^{2}k\left(  t,\tau\right)  }{\partial t^{2}}%
}dt-\frac{l\left(  l+1\right)  }{x^{2}}k\left(  x,\tau\right)  -\int_{0}%
^{x}{K\left(  x,t\right)  \frac{l\left(  l+1\right)  }{t^{2}}k\left(
t,\tau\right)  }dt.
\]
Integration by parts leads to the equality
\begin{align*}
\int_{0}^{x}{K\left(  x,t\right)  \frac{\partial^{2}k\left(  t,\tau\right)
}{\partial t^{2}}}dt  &  =K\left(  x,x\right)  {\frac{\partial k\left(
t,\tau\right)  }{\partial t}}\Big|_{t=x}-K\left(  x,0\right)  {\frac{\partial
k\left(  t,\tau\right)  }{\partial t}}\Big|_{t=0}\\
& \quad-\frac{\partial K\left(  x,t\right) }{\partial t}\Big|_{t=x}k\left(
x,\tau\right)  +\frac{\partial K\left(  x,t\right)  }{\partial t}%
\Big|_{t=0}k\left(  0,\tau\right)  +\int_{0}^{x}{\frac{\partial^{2}K\left(
x,t\right)  }{\partial t^{2}}}k\left(  t,\tau\right)  dt.
\end{align*}
Taking into account (\ref{K(x,0)=0}) and the fact that
\begin{equation}
k\left(  0,\tau\right)  =\mathbf{T}^{-1}\left[  K\left(  x,\tau\right)
\right]  \Big|_{x=0}=K\left(  0,\tau\right)  =0 \label{k(0,tau)}%
\end{equation}
we obtain%
\[
\int_{0}^{x}{K\left(  x,t\right)  \frac{\partial^{2}k\left(  t,\tau\right)
}{\partial t^{2}}}dt=K\left(  x,x\right)  {\frac{\partial k\left(
x,\tau\right)  }{\partial x}}-\frac{\partial K\left(  x,t\right)  }{\partial
t}\Big|_{t=x}k\left(  x,\tau\right)  +\int_{0}^{x}{\frac{\partial^{2}K\left(
x,t\right)  }{\partial t^{2}}}k\left(  t,\tau\right)  dt
\]
and consequently,
\begin{multline*}
\mathbf{L}\mathbf{T}\left[  k\left(  x,\tau\right)  \right]  -\mathbf{TL}%
_{0}\left[  k\left(  x,\tau\right)  \right] \\
=\int_{0}^{x}{\left[  \frac{\partial^{2}K\left(  x,t\right)  }{\partial x^{2}%
}-\frac{\partial^{2}K\left(  x,t\right)  }{\partial t^{2}}-\frac{l\left(
l+1\right)  }{x^{2}}K\left(  x,t\right)  +\frac{l\left(  l+1\right)  }{t^{2}%
}K\left(  x,t\right)  -q\left(  x\right)  K\left(  x,t\right)  \right]
k\left(  x,\tau\right)  }dt\\
+2\frac{dK\left(  x,x\right)  }{dx}k\left(  x,\tau\right)  -q\left(  x\right)
k\left(  x,\tau\right)  =0.
\end{multline*}

\end{Proof}

In the following Proposition we summarize properties of $k$.

\begin{Proposition}
\label{Prop Goursat problen k}The function $k\left(  x,\tau\right)  $ defined
by \eqref{K} satisfies the following properties

\begin{enumerate}
\item on $0<x<\tau<b$ it satisfies the equation
\begin{equation}
\left(  \square-\frac{l\left(  l+1\right)  }{x^{2}}+\frac{l\left(  l+1\right)
}{\tau^{2}}\right)  k\left(  x,\tau\right)  =0. \label{Wave sing no q}%
\end{equation}

\item $\frac{dk\left(  x,x\right)  }{dx}\in C\left(  \left[  0,b\right]
\right)  $.

\item
\begin{equation}
k\left(  0,\tau\right)  =0. \label{k(0,t)}%
\end{equation}

\end{enumerate}
\end{Proposition}

\begin{Proof}
Property 1. follows directly from Lemma \ref{Lemma transmutation of k}. From
(\ref{k}) we have that $k\left(  x,x\right)  =K\left(  x,x\right)  -\int
_{0}^{x}{K^{2}\left(  t,x\right)  }dt$. Then
\begin{equation}
\frac{dk\left(  x,x\right)  }{dx}=\frac{1}{2}q\left(  x\right)  -\frac{d}%
{dx}\int_{0}^{x}{K^{2}\left(  t,x\right)  }dt. \label{dk/dx}%
\end{equation}
This expression is a continuous function because $q$ and $K$ are continuous.
Finally, Property 3. was established above (\ref{k(0,tau)}).
\end{Proof}

\section{Generalized wave polynomials}

\subsection{Wave polynomials}

\begin{Definition}
[\cite{KhmeKravTorTrem}] The polynomials
\begin{equation*}
p_{0}(x,t)=1,\quad p_{2m-1}(x,t)=\sum_{\mathrm{even}\text{ }k=0}^{m}\binom
{m}{k}x^{m-k}t^{k},\quad p_{2m}(x,t)=\sum_{\mathrm{odd}\text{ }k=1}^{m}%
\binom{m}{k}x^{m-k}t^{k},%\label{Eq1.20}%
\end{equation*}
are called \textbf{wave polynomials}.
\end{Definition}

Every wave polynomial is a solution of the wave equation $\square u(x,t)=0$.
Our next aim is to use them for constructing a system of solutions of equation
(\ref{Intro KG-sing0}). For this together with the operator $Y_{l,x} $ we
consider the operator $Y_{l,t}$. Due to Theorem \ref{Teorema 1} a solution
$u(x,t)$ of the wave equation is transformed into a solution of
(\ref{Intro KG-sing0}) by the composition of the operators $Y_{l,x}Y_{l,t}$
whenever $\frac{\partial u(x,t)}{\partial x}\Big|_{x=0}=\frac{\partial
u(x,t)}{\partial t}\Big|_{t=0}=0$. Indeed, if this is true, from (\ref{Eq.T})
we obtain
\[
\left(  \square-\frac{l\left(  l+1\right)  }{x^{2}}+\frac{l\left(  l+1\right)
}{t^{2}}\right)  Y_{l,x}Y_{l,t}u(x,t)=Y_{l,x}Y_{l,t}\square u(x,t)=0
\]
and hence $U(x,t):=Y_{l,x}Y_{l,t}u(x,t)$ is a solution of
(\ref{Intro KG-sing0}).

\begin{Remark}
The wave polynomials satisfying the condition $\frac{\partial u(x,t)}{\partial
x}\Big|_{x=0}=\frac{\partial u(x,t)}{\partial t}\Big|_{t=0}=0$ are
$p_{0}(x,t)$ and
\begin{equation*}
p_{4n-1}\left(  x,t\right)  =\sum_{\text{even }k=0}^{2n}{{\binom{2n}{k}%
}x^{2n-k}t^{k}},\qquad n=1,2,\ldots. %\label{EP}%
\end{equation*}
\end{Remark}

\subsection{Generalized wave polynomials related to equation
(\ref{Intro KG-sing0})}

The images of the wave polynomials under the action of the operator
$Y_{l,x}Y_{l,t}$ are called generalized wave polynomials related to
(\ref{Intro KG-sing0}).

\begin{Proposition}
\label{Prop Gen Polyn U}The generalized wave polynomials $U_{0}\left(
x,t\right)  =Y_{l,x}Y_{l,t}\left[  p_{0}\left(  x,t\right)  \right]  $ and
$U_{4n-1}\left(  x,t\right)  = \linebreak[4]Y_{l,x}Y_{l,t}[ p_{4n-1}\left(
x,t\right)  ] $ satisfy equation \eqref{Intro KG-sing0} together with the
boundary conditions
\begin{equation}
U\left(  0,t\right)  =0\qquad\text{and}\qquad U\left(  x,0\right)  =0.
\label{boundary cond U}%
\end{equation}
They have the form
\begin{equation}
U_{0}\left(  x,t\right)  =C_{0}^{2}x^{l+1}t^{l+1} \label{U0}%
\end{equation}
and
\begin{equation}
\displaystyle U_{4n-1}\left(  x,t\right)  =\sum_{\text{even }k=0}^{2n}%
{{\binom{2n}{k}}\left[  \frac{\Gamma{\left(  n+\frac{1-k}{2}\right)  }%
\Gamma{\left(  \frac{k+1}{2}\right)  }\Gamma^{2}{\left(  l+1\right)  }%
}{2^{2\left(  l+\frac{3}{2}\right)  }\Gamma^{2}{\left(  l+\frac{3}{2}\right)
}\Gamma{\left(  n+l+\frac{3-k}{2}\right)  }\Gamma{\left(  l+\frac{k+3}%
{2}\right)  }}x^{2n-k+l+1}t^{k+l+1}\right]  }. \label{Eq1.21}%
\end{equation}

\end{Proposition}

\begin{Proof}
The generalized polynomials $U_{0}$, $U_{4n-1}$, $n=1,2,\ldots$ are solutions
of (\ref{Intro KG-sing0}) due to the explanation given in the preceding
subsection. Equalities (\ref{boundary cond U}) follow from the definition of
$Y_{l,x}$ and $Y_{l,t}$. Finally, (\ref{U0}) and (\ref{Eq1.21}) follow from
Proposition \ref{P1.1}.\bigskip
\end{Proof}

\subsection{Generalized wave polynomials related to equation (\ref{waveq})}

Consider the functions
\begin{equation}
u_{0}\left(  x,t\right)  =\mathbf{T}\left[  U_{0}\left(  x,t\right)  \right]
\qquad\text{and}\qquad u_{n}\left(  x,t\right)  =\mathbf{T}\left[
U_{4n-1}\left(  x,t\right)  \right]  ,\quad n=1,2,\ldots. \label{u_n}%
\end{equation}
They will be called the generalized wave polynomials related to equation
(\ref{waveq}). The following statement summarizes some important properties of
these functions.

\begin{Proposition}
\begin{enumerate}
\item The functions \eqref{u_n} are solutions of \eqref{waveq}.

\item The functions \eqref{u_n} can be written as follows
\begin{equation}
u_{n}\left(  x,t\right)  =u_{0}\left(  x\right)  \sum_{k=0}^{n}{\Xi
_{n,k}{\tilde{X}}^{\left(  2\left(  n-k\right)  \right)  }\left(  x\right)
t^{2k+l+1}},\qquad n=1,2,\ldots\label{u_n explicit}%
\end{equation}
where $\Xi_{n,k}=C^{2}{\binom{2n}{2k}}\left(  -1\right)  ^{n-k}2^{2\left(
n-k\right)  }\left(  n-k\right)  !\frac{\Gamma{\left(  n+\frac{1}{2}-k\right)
}\Gamma{\left(  k+\frac{1} {2}\right)  }}{\Gamma{\left(  k+l+\frac{3}%
{2}\right)  }\Gamma{\left(  l+\frac{3}{2}\right)  }}$, $C=\frac{\Gamma{\left(
l+1\right)  }} {2^{l+\frac{3}{2}}\Gamma{\left(  l+\frac{3}{2}\right)  }}$.

\item The functions \eqref{u_n} satisfy the properties
\begin{align*}
\lim_{t\rightarrow0}{u_{n}\left(  x,t\right)  \cdot t^{l}} & =0,\qquad\forall
n,\ \forall l\in\left(  -\frac{1}{2},0\right)  \qquad\text{and}\qquad
u_{n}\left(  x,0\right)  =0,\qquad\forall n,\ \forall l\geq-\frac{1}{2},\\%\label{3.12}\\
\lim_{x\rightarrow0}{u_{n}\left(  x,t\right)  \cdot x^{l}} & =0, \qquad\forall
n,\ \forall l\in\left(  -\frac{1}{2},0\right)  \qquad\text{and}\qquad
u_{n}\left(  0,x\right)  =0,\qquad\forall n,\ \forall l\geq-\frac{1}{2}.%\label{3.12B}%
\end{align*}
\end{enumerate}
\end{Proposition}

\begin{Proof}
Property 1. follows from Theorem \ref{Th Transmutation property}. Property 2.
is a direct consequence of Proposition \ref{Prop Mapping property}. Finally,
property 3. follows from (\ref{u_n explicit}).
\end{Proof}

Let us introduce a special notation for the traces of the generalized wave
polynomials (\ref{u_n}) on the line $t=x$. Thus,
\begin{equation}\label{cn def}
\mathbf{c}_{n}\left(  x\right)  :=\frac{\Gamma^{2}{\left(  l+1\right)  }\pi
}{2^{2\left(  l+\frac{3}{2}\right)  }\Gamma^{3}{\left(  l+\frac{3}{2}\right)
}}u_{0}\left(  x\right)  \sum_{k=0}^{n}{\frac{\left(  -1\right)  ^{n-k}\left(
2n\right)  !}{4^{k}\Gamma{\left(  k+1\right)  }\Gamma{\left(  k+l+\frac{3}%
{2}\right)  }}}\widetilde{X}^{\left(  2\left(  n-k\right)  \right)  }\left(
x\right)  x^{2k+l+1},\quad n=1,2,\ldots.
\end{equation}

\section{Analytical approximation of the transmutation kernel}\label{Sect 5}

The aim of this section is to show that the transmutation kernel $K$ can be
uniformly approximated by linear combinations of the generalized wave
polynomials (\ref{u_n explicit}). For this we first consider its preimage $k$
and prove that it can be uniformly approximated by the generalized wave
polynomials (\ref{U0}), (\ref{Eq1.21}). This is done by establishing the
well-posedness of a corresponding Goursat problem appearing in Proposition
\ref{Prop Goursat problen k} and certain completeness of the traces of
(\ref{U0}), (\ref{Eq1.21}) on $t=x$.

The following lemma is a simple corollary of the Müntz theorem \cite[Chap.
11]{DeVoreLorentz}, \cite[Chap. XI]{Philip}.

\begin{Lemma}
\label{Lemma Muntz}Let $P_{j}$ be a sequence of different positive numbers,
such that $\lim_{j\rightarrow\infty}{P_{j}}=\infty$ and $\sum_{j=1}^{\infty
}{\frac{1}{P_{j}}}=\infty$. Then $\left\{  x^{P_{j}}\right\}  $ is complete in
$C_{0}\left(  \left[  0,b\right]  \right)  $. If additionally $P_{j}>1$ for
all $j\in\mathbf{N}$, then $\left\{  x^{P_{j}}\right\}  $ is complete in
$C_{0}^{1}\left(  \left[  0,b\right]  \right)  $.
\end{Lemma}

Now we are interested to show that the function $\frac{dk\left(  x,x\right)
}{dx}$ can be approximated uniformly by linear combinations of the functions
$\frac{dU_{0}\left(  x,x\right)  }{dx}$ and $\frac{dU_{4n-1}\left(
x,x\right)  }{dx}$, $n=1,2,\ldots$. We observe that $\frac{dU_{0}\left(
x,x\right)  }{dx}=\bar{C}_{0}x^{2l+1}$ and $\frac{dU_{4n-1}\left(  x,x\right)
}{dx}=\bar{C}_{n}x^{2\left(  n+l\right)  +1}$, $n=1,2,\ldots$ where $\bar
{C}_{n}$ are constants, and hence the zero power of $x$ is an element of the
set $\left\{  \frac{dU_{0}\left(  x,x\right)  }{dx}\right\}  \cup\left\{
\frac{dU_{4n-1}\left(  x,x\right)  }{dx}\right\} _{n=1}^{\infty}$ if only
$l=-\frac{1}{2}$. However, as we notice below (Proposition
\ref{Prop k(x,x) approx}) and can be seen from (\ref{dk/dx}), if
\begin{equation}
q(0)=0 \label{q0=0}%
\end{equation}
the function $k\left(  x,x\right)  $ is an element of $C_{0}^{1}\left(
\left[  0,b\right]  \right)  $. Without loss of generality one can always
suppose that (\ref{q0=0}) is fulfilled. Indeed, considering $\Lambda
:=\lambda-q\left(  0\right)  $ and $q_{0}\left(  x\right)  :=q\left(
x\right)  -q\left(  0\right)  $ one can rewrite equation
(\ref{Intro Bessel perturbed}) in the form
\begin{equation}\label{Bessel perturbed transformed}
\left(  -\frac{d^{2}}{dx^{2}}+\frac{l\left(  l+1\right)  }{x^{2}}+q_{0}\left(
x\right)  \right)  u=\Lambda u,
\end{equation}
where the potential already fulfills (\ref{q0=0}). From now on we assume
(\ref{q0=0}) to be fulfilled. Then the following result is valid.

\begin{Proposition}
\label{Prop k(x,x) approx}For any $\varepsilon>0$ there exist such
$N\in\mathbb{N}$ and $a_{0},a_{1},\ldots,a_{N}\in\mathbb{C}$ that for all
$x\in\left[  0,b\right]  $ the inequality holds
\begin{equation*}
\left\vert \frac{dk\left(  x,x\right)  }{dx}-\frac{dk_{N}\left(  x,x\right)
}{dx}\right\vert \leq\varepsilon%\label{3.16}%
\end{equation*}
where
\begin{equation}
k_{N}\left(  x,t\right)  :=\displaystyle{a_{0}U_{0}\left(  x,t\right)  +}%
\sum_{\kappa=1}^{N}{a_{\kappa}U_{4\kappa-1}\left(  x,t\right)  }. \label{kN}%
\end{equation}

\end{Proposition}

\begin{Proof}
Due to Lemma \ref{Lemma Muntz} it is sufficient to show that $k\left(
x,x\right)  \in C_{0}^{1}\left(  \left[  0,b\right]  \right)  $, since the
traces of the generalized wave polynomials ${U_{0}\left(  x,t\right)  }$,
${U_{4\kappa-1}\left(  x,t\right)  }$, $k=1,2,\ldots$ on the line $t=x$
satisfy the conditions of Lemma \ref{Lemma Muntz}. From Property 3. of
Proposition \ref{Prop Goursat problen k} we have that $k\left(  0,0\right)
=0$. Moreover, from (\ref{dk/dx}) we obtain that $\frac{dk\left(  0,0\right)
}{dx}=\frac{1}{2}q\left(  0\right)  =0$. Thus, indeed $k\left(  x,x\right)
\in C_{0}^{1}\left(  \left[  0,b\right]  \right)  $, and the proof follows
from Lemma \ref{Lemma Muntz}.
\end{Proof}

Next step is to prove that the uniform approximation of the preimage of the
transmutation kernel on the line $t=x$ implies its uniform approximation on
the triangle $0\leq x\leq t\leq b$. This means the stability of the Goursat
problem appearing in Proposition \ref{Prop Goursat problen k}. First, in the
following statement, using the result from \cite{Volk} we write down its solution.

\begin{Proposition}
The unique solution in the triangle $0< x\leq t\leq b$ of the Goursat problem
\begin{gather}
\left(  \square-\frac{l\left(  l+1\right)  }{x^{2}}+\frac{l\left(  l+1\right)
}{t^{2}}\right)  H\left(  x,t\right)  =0,\qquad l\geq-\frac{1}{2}%
\label{3.18}\\
\frac{dH\left(  x,x\right)  }{dx}=\frac{h\left(  x\right)  }{2},\label{3.19}\\
H\left(  0,t\right)  =0,\label{3.19B}%
\end{gather}
where $h\left(  x\right)  $ is a continuous function, is given by the formula
\begin{equation}
H\left(  x,t\right)  =\left(  z-s\right)  ^{-l}u\left(  z,s\right)  ,
\label{3.20}%
\end{equation}
with $z=\frac14(x+t)^{2}$, $s=\frac14(t-x)^{2}$ and $u$ being defined by
\begin{equation}
u\left(  \xi,\eta\right)  =\frac{1}{4}\int_{0}^{\xi}{v\left(  z,0;\xi
,\eta\right)  h\left(  \sqrt{z}\right)  z^{l-\frac{1}{2}}}dz \label{3.21}%
\end{equation}
where $v$ is the Riemann function of the problem admitting an explicit
representation given in Appendix \ref{Appendix}.
\end{Proposition}

\begin{Proof}
The change of variables $z=\frac{1}{4}\left(  x+t\right)  ^{2}$, $s=\frac
{1}{4}\left(  t-x\right)  ^{2}$ transforms (\ref{3.18})--(\ref{3.19B}) into
the following boundary value problem for $u\left(  z,s\right) $
\begin{gather}
\frac{\partial^{2}u}{\partial z\partial s}+\frac{l}{z-s}\frac{\partial
u}{\partial z}-\frac{l}{z-s}\frac{\partial u}{\partial s}=0,\label{3.22}\\
\frac{\partial u}{\partial z}-\frac{lu}{z}=\frac{1}{4}h\left(  \sqrt
{z}\right)  z^{l-\frac{1}{2}},\qquad s=0,\label{3.23}\\
u\left(  z,z\right)  =0. \label{3.24}%
\end{gather}
We denote by $Z$ the domain of this boundary problem, given by $0\le s\le z\le
b^{2}$, $\sqrt{s}+\sqrt{z}\le b$. Note that (\ref{3.23}) is obtained from
(\ref{3.19}). Indeed, differentiation of $H\left(  x,x\right)  =z^{-l}u\left(
z,0\right)  $ with respect to $x$ gives
\[
\frac{d}{dz}\left[  z^{-l}u\left(  z,0\right)  \right]  \frac{dz}{dx}%
=-2\sqrt{z}\left[  lz^{-\left(  l+1\right)  }u\left(  z,0\right)  -z^{-l}%
\frac{\partial u\left(  z,0\right)  }{\partial z}\right]  =-\frac{2lu\left(
z,0\right)  }{z^{l+\frac{1}{2}}}+\frac{2}{z^{l-\frac{1}{2}}}\frac{\partial
u\left(  z,0\right)  }{\partial z}.
\]
On the other hand, $h\left(  x\right)  =h\left(  \sqrt{z}\right)  $. Thus,
\[
\frac{1}{z^{l-\frac{1}{2}}}\frac{\partial u\left(  z,0\right)  }{\partial
z}-\frac{lu\left(  z,0\right)  }{z^{l+\frac{1}{2}}}=\frac{1}{4}h\left(
\sqrt{z}\right)  .
\]
Multiplication by $z^{l-\frac{1}{2}}$ gives us (\ref{3.23}). (\ref{3.24})
follows directly from (\ref{3.19B}). Solution of (\ref{3.22})--(\ref{3.24}) by
the Riemann function method (see \cite{Volk} for details) leads to the
integral (\ref{3.21}).
\end{Proof}

In the following auxiliary statement the stability of the Goursat problem
(\ref{3.18}), (\ref{3.19}) is established.

\begin{Proposition}
\label{Prop Stability Goursat}There exists a constant $C>0$ such that if $H$
is a solution of \eqref{3.18}--\eqref{3.19B} with $h$ in \eqref{3.19}
satisfying $\left\Vert h\right\Vert _{C\left(  \left[  0,b\right]  \right)
}\leq\varepsilon$, then
\[
\max_{0\leq x\leq t\leq b}\left\vert H\left(  x,t\right)  \right\vert
\leq\varepsilon C.
\]

\end{Proposition}

\begin{proof}
Consider the solution $H$ of the Goursat problem defined in terms of $u$ by
(\ref{3.20}).  Due to Corollary \ref{Estimates Riemann 2} from Appendix \ref{Appendix} we have that for $\left\Vert
h\right\Vert _{C\left(  \left[  0,b\right]  \right)  }\leq\varepsilon$, the
function $u$ defined by (\ref{3.21}) can be bounded as follows
\begin{equation*}
\left\vert u\left(  \xi,\eta\right)  \right\vert \leq\frac{\varepsilon(A_1+A_2)}4\left(
\xi-\eta\right)  ^{l+1/2}.%\label{Cepsilon}%
\end{equation*}
Denote by $A:=(A_{1}+A_{2})/4$. Hence from (\ref{3.20}) we obtain
\[
\begin{split}
\max_{0\leq x\leq t\leq b}\left\vert H\left(  x,t\right)  \right\vert
&=\max_{(z,s)\in Z}\vert \left(  z-s\right)
^{-l} u\left(  z,s\right)  \vert
\leq\max_{(z,s)\in Z}\left\{  A\varepsilon| \left(  z-s\right)
^{-l}| \left(  z-s\right)  ^{l+1/2}\right\}\\
&\le\varepsilon A\max_{0\le s\le z \le b^2}\left(  z-s\right)
^{1/2}=\varepsilon A b.
\end{split}
\]
\end{proof}

\begin{Proposition}
\label{Prop Approx k}Let $\varepsilon>0$. There exist such $N\in\mathbb{N}$
and $a_{0},a_{1},\ldots,a_{N}\in\mathbb{C}$ that the inequality
\[
\left\vert k\left(  x,t\right)  -k_{N}\left(  x,t\right)  \right\vert
\leq\varepsilon
\]
is fulfilled for all $0\le x\le t\le b$, where $k_{N}\left(  x,t\right)  $ is
of the form \eqref{kN}.
\end{Proposition}

\begin{proof}
Take $\varepsilon_{1}:=\varepsilon/C$ where $C$ is the constant from
Proposition \ref{Prop Stability Goursat}. Due to Proposition
\ref{Prop k(x,x) approx} there exist $N\in\mathbb{N}$ and $a_{0},a_{1}%
,\ldots,a_{N}$ such that $\left\vert \frac{dU_{\varepsilon}\left(  x,x\right)
}{dx}\right\vert \leq\varepsilon_{1}$ where  $U_{\varepsilon}\left(
x,t\right)  :=k\left(  x,t\right)  -k_{N}\left(  x,t\right)  $.  Since by
Proposition 4.1 every generalized wave polynomial $U_{0}\left(  x,t\right)  $,
$U_{4n-1}\left(  x,t\right)  $, $n=1,2,\ldots$ satisfies equation
(\ref{Wave sing no q}), and $U_{0}\left(  0,t\right)  =0$, $U_{4n-1}\left(
0,t\right)  =0$, $n=1,2,\ldots$,   $U_{\varepsilon}\left(  x,t\right)  $ also
satisfies equation (\ref{Wave sing no q}) and boundary condition
(\ref{k(0,t)}). Due to Proposition \ref{Prop Stability Goursat} we obtain
$\max_{0\leq x\leq t\leq b}\left\vert U_{\varepsilon}\left(  x,t\right)
\right\vert \leq\varepsilon$.
\end{proof}

\subsection{Approximation of the integral kernel $K\left(  x,t\right)  $}

\begin{Theorem}
\label{Th Approx K}For any $\varepsilon>0$ there exist such $N\in\mathbb{N}$
and $a_{0},a_{1},\ldots,a_{N}\in\mathbb{C}$ that the inequality
\begin{equation*}
\left\vert K\left(  x,t\right)  -K_{N}\left(  x,t\right)  \right\vert
\leq\varepsilon,%\label{3.28}%
\end{equation*}
holds for all $0\leq x\leq t\leq b$, where
\begin{equation}
K_{N}\left(  x,t\right)  :=\sum_{\kappa=0} ^{N}{a_{\kappa}u_{\kappa}\left(
x,t\right)  .}\label{3.27}%
\end{equation}
In particular,
\begin{equation}
\left\vert \frac{1}{2}\int_{0}^{x}{q\left(  s\right)  }ds-\sum_{\kappa=0}
^{N}{a_{\kappa}}\mathbf{c}_{\kappa}\left(  x\right)  \right\vert
\leq\varepsilon\label{Q approx}%
\end{equation}
for all $x\in\left[  0,b\right]  $.
\end{Theorem}

\begin{Proof}
Let $\left\Vert \mathbf{T}\right\Vert $ denote the uniform norm of the
operator $\mathbf{T}$. Choose $\varepsilon>0$ \ and consider $\varepsilon
_{1}:=\varepsilon/\left\Vert \mathbf{T}\right\Vert $. Due to Proposition
\ref{Prop Approx k} there exist such $N\in\mathbb{N}$ and $a_{0},a_{1}
,\ldots,a_{N}\in\mathbb{C}$ that the inequality
\[
\left\Vert k-k_{N}\right\Vert :=\max_{0\leq x\leq t\leq b}\left\vert k\left(
x,t\right)  -k_{N}\left(  x,t\right)  \right\vert \leq\varepsilon_{1}
\]
holds. Consider
\[
\left\Vert K-K_{N}\right\Vert =\left\Vert \mathbf{T}\left(  k-k_{N}\right)
\right\Vert \leq\left\Vert \mathbf{T}\right\Vert \left\Vert k-k_{N}\right\Vert
\leq\varepsilon_{1}\left\Vert \mathbf{T}\right\Vert =\varepsilon.
\]
Inequality (\ref{Q approx}) follows from the last inequality as well if one
takes into account that $K\left(  x,x\right)  =\frac{1}{2}\int_{0}
^{x}{q\left(  s\right)  }ds$.
\end{Proof}

\section{Approximation of the solution}

According to Theorem \ref{Th Approx K} for any $\varepsilon>0$ there exist
such $N\in\mathbb{N}$ and $a_{0},a_{1},\ldots,a_{N}\in\mathbb{C}$ that the
inequality (\ref{Q approx}) holds. We will assume that the coefficients
$a_{0},a_{1},\ldots,a_{N}$ \ are thus chosen and denote by $K_{N}$ the
corresponding approximate transmutation kernel (\ref{3.27}). Denote
$q_{N}(x):=\frac{2dK_{N}(x,x)}{dx}$ . Note that (\ref{Q approx}) implies the
inequality
\[
\left\vert \int_{0}^{x}\left(  q(t)-q_{N}(t)\right)  dt\right\vert
\leq\varepsilon.
\]
We are interested in the corresponding approximate solution
\begin{equation}
\mathbf{u}_{N}\left(  x,\lambda\right)  :=\mathbf{d}_{l}\left(  x,\lambda
\right)  +\int_{0}^{x}{K_{N}\left(  x,t\right)  \mathbf{d}_{l}\left(
t,\lambda\right)  }dt,\qquad l\ge-1/2. \label{uN}%
\end{equation}

The following statement allows one to estimate the residual of the approximate solution.

\begin{Theorem}\label{Thm Approx Sol}
For the approximate solution \eqref{uN} the following inequality is valid
\begin{equation*}
\left\vert \int_{0}^{x}{\left(  \mathbf{L}+\lambda\right)  \mathbf{u}
_{N}\left(  s\right)  }ds\right\vert \leq2\varepsilon\sqrt{x}.%\label{residual}%
\end{equation*}
\end{Theorem}

\begin{Proof}
First, let us show that
\[
\left(  \mathbf{L}+\lambda\right)  \mathbf{u}_{N}\left(  x\right)  =\left[
q_{N}\left(  x\right)  -q\left(  x\right)  \right]  \mathbf{d}_{l}\left(
x,\lambda\right)  .
\]
Indeed,
\begin{align*}
\left(  \mathbf{L}+\lambda\right)  \mathbf{u}_{N}\left(  x\right)   &
=-q\left(  x\right)  \mathbf{d}_{l}\left(  x,\lambda\right)  +\int_{0}
^{x}\left(  \left(  \mathbf{L}+\lambda\right)  {K_{N}\left(  x,t\right)
}\right)  {\mathbf{d}_{l}\left(  t,\lambda\right)  }dt\\
&  \quad+\frac{\partial K_{N}\left(  x,t\right)  }{\partial x}\Big|_{t=x}
\mathbf{d}_{l}\left(  x,\lambda\right)  +\frac{dK_{N}\left(  x,x\right)  }
{dx}\mathbf{d}_{l}\left(  x,\lambda\right)  +K_{N}\left(  x,x\right)
\mathbf{d}_{l}^{\prime}\left(  x,\lambda\right)  .
\end{align*}
Since $K_{N}$ satisfies equation (\ref{waveq}) we have
\[
\int_{0}^{x}\left(  {\mathbf{L}K_{N}\left(  x,t\right)  }\right)
{\mathbf{d}_{l}\left(  t,\lambda\right)  }dt=\int_{0}^{x}{\left(
\frac{\partial^{2}K_{N}\left(  x,t\right)  }{\partial t^{2}}-\frac{l\left(
l+1\right)  }{t^{2}}K_{N}\left(  x,t\right)  \right)  \mathbf{d}_{l}\left(
t,\lambda\right)  }dt.
\]
Integration by parts gives
\begin{align*}
\int_{0}^{x}{\frac{\partial^{2}K_{N}\left(  x,t\right)  }{\partial t^{2}
}\mathbf{d}_{l}\left(  t,\lambda\right)  }dt  &  =\frac{\partial K_{N}\left(
x,t\right)  }{\partial t}\Big|_{t=x}\mathbf{d}_{l}\left(  x,\lambda\right)
-K_{N}\left(  x,x\right)  \mathbf{d}_{l}^{\prime}\left(  x,\lambda\right)
-\frac{\partial K_{N}\left(  x,t\right)  }{\partial t}\mathbf{d}_{l}\left(
t,\lambda\right) \Big|_{t=0}\\
&  +K_{N}\left(  x,t\right)  \mathbf{d}_{l}^{\prime}\left(  t,\lambda\right)
\Big|_{t=0} +\int_{0}^{x}{K_{N}\left(  x,t\right)  \mathbf{d}_{l}%
^{\prime\prime}\left(  t,\lambda\right)  }dt.
\end{align*}
Note that $K_{N}(x,t)$ has the form $t^{l+1}\sum_{n=0}^{N}\kappa_{n}(x)t^{2n}%
$, while $\mathbf{d}_{l}\left(  t,\lambda\right) \sim ct^{l+1}$ and
$\mathbf{d}^{\prime}_{l}\left(  t,\lambda\right) \sim c(l+1)t^{l}$, $t\to0$,
with $c=\frac{\lambda^{(2l+1)/4}}{\Gamma(l+3/2)2^{l+1/2}}$, leading to the
conclusion that
\[
-\frac{\partial K_{N}\left(  x,t\right)  }{\partial t}\mathbf{d}_{l}\left(
t,\lambda\right) \Big|_{t=0}+K_{N}\left(  x,t\right)  \mathbf{d}_{l}^{\prime
}\left(  t,\lambda\right) \Big|_{t=0} = 0
\]
for any $l\ge-1/2$. Thus,
%$K_{N}\left(
%x,0\right)  \mathbf{d}_{l}^{\prime}\left(  0,\lambda\right)  =0$ because
%$\lim_{t\rightarrow0}{K_{N}\left(  x,t\right)  t^{l}}=0$ for all $l>-\frac
%{1}{2}$ and $\mathbf{d}_{l}^{\prime}\left(  t,\lambda\right)  =O\left(
%t^{l}\right)  $.%
\[
\int_{0}^{x}{\frac{\partial^{2}K_{N}\left(  x,t\right)  }{\partial t^{2}%
}\mathbf{d}_{l}\left(  t,\lambda\right)  }dt=\frac{\partial K_{N}\left(
x,t\right)  }{\partial t}\Big|_{t=x}\mathbf{d}_{l}\left(  x,\lambda\right)
-K_{N}\left(  x,x\right)  \mathbf{d}_{l}^{\prime}\left(  x,\lambda\right)
+\int_{0}^{x}{K_{N}\left(  x,t\right)  \mathbf{d}_{l}^{\prime\prime}\left(
t,\lambda\right)  }dt.
\]
Hence
\begin{align*}
\left(  \mathbf{L}+\lambda\right)  \mathbf{u}_{N}\left(  x\right)   &
=\left(  2\frac{dK_{N}\left(  x,x\right)  }{dx}-q\left(  x\right)  \right)
\mathbf{d}_{l}\left(  x,\lambda\right)  +\int_{0}^{x}\left(  {\frac
{d^{2}\mathbf{d}_{l}\left(  t,\lambda\right)  }{dt^{2}}-\frac{l\left(
l+1\right)  }{t^{2}}\mathbf{d}_{l}\left(  t,\lambda\right)  +\lambda
\mathbf{d}_{l}\left(  t,\lambda\right)  }\right)  {K_{N}\left(  x,t\right)
}dt\\
&  =\left(  q_{N}\left(  x\right)  -q\left(  x\right)  \right)  \mathbf{d}
_{l}\left(  x,\lambda\right)  .
\end{align*}
Consider
\[
\int_{0}^{x}\left(  q_{N}\left(  s\right)  -q\left(  s\right)  \right)
{\mathbf{d}_{l}\left(  s,\lambda\right)  }ds=\mathbf{d}_{l}\left(
x,\lambda\right)  \int_{0}^{x}\left(  {q_{N}\left(  s\right)  -q\left(
s\right)  }\right)  ds-\int_{0}^{x}{\mathbf{d}_{l}^{\prime}\left(
s,\lambda\right)  \int_{0}^{s}}\left(  {{q_{N}\left(  \sigma\right)  -q\left(
\sigma\right)  }}\right)  {d\sigma}ds
\]
Since $\left\vert \mathbf{J}_{\nu}\left(  x\right)  \right\vert \leq1$ for any
$\nu\geq0$ and all $x\in\mathbb{R}$ (see e.g., \cite[10.14.1.]{OlLoBoiCla}) we
have $\left\vert \mathbf{d}_{l}\left(  x,\lambda\right)  \right\vert \leq
\sqrt{x}$. Hence
\[
\left\vert \int_{0}^{x}{\left(  \mathbf{L}+\lambda\right)  \mathbf{u}
_{N}\left(  s\right)  }ds\right\vert =\left\vert \int_{0}^{x}\left(
{q_{N}\left(  s\right)  -q\left(  s\right)  }\right)  {\mathbf{d}_{l}\left(
s,\lambda\right)  }ds\right\vert \leq\varepsilon\left(  \sqrt{x}+\left\vert
\int_{0}^{x}{\mathbf{d}_{l}^{\prime}\left(  s,\lambda\right)  }ds\right\vert
\right)  \leq2\varepsilon\sqrt{x}.
\]

\end{Proof}

The convenience of seeking the approximate solution in the form (\ref{uN})
consists in the fact that the integral in (\ref{uN}) can be calculated
explicitly. Consider
\[
\mathbf{u}_{N}\left(  x\right)  =\mathbf{d}_{l}\left(  x,\lambda\right)
+\int_{0}^{x}\sum_{n=0} ^{N}{a_{n}u_{n}\left(  x,t\right)
}{\mathbf{d}_{l}\left(  t,\lambda\right)  }\,dt.
\]
By (\ref{u_n explicit}) we have
\begin{equation}\label{uN expanded}
\mathbf{u}_{N}\left(  x\right)  =\mathbf{d}_{l}\left(  x,\lambda\right)
+u_{0}\left(  x\right)  \sum_{n=0}^{N}{a_{n}\sum_{k=0}^{n}{\Xi_{n,k}%
\widetilde{X}^{\left(  2\left(  n-k\right)  \right)  }\left(  x\right)
\int_{0}^{x}{t^{2k+l+1}}\mathbf{d}_{l}\left(  t,\lambda\right)  }}\,dt.
\end{equation}
The integrals here can be calculated explicitly. Let $\lambda = \omega^2$. In an important special case $l\in{\mathbb{N}}$ one can use the following formula \cite[(1.8.1.6)]
{PrudBriMar}
\begin{align*}
\int{t^{\left(  2k+l+1\right)  +\frac{1}{2}}\mathbf{J}_{l+\frac{1}{2}}\left(
\omega t\right)  }dt  &  =\frac{1}{\omega^{2k+l+\frac{5}{2}}}\sqrt{\frac
{2}{\pi}}\sin\left(  \omega t-\frac{l\pi}{2}\right)  \sum_{j=0}^{\left[
\frac{2k+l}{2}\right]  }{C_{2k+l-2j}\left(  \omega t\right)  ^{2k+l-2j}}\\
&  \quad+\frac{1}{\omega^{2k+l+\frac{5}{2}}}\sqrt{\frac{2}{\pi}}\cos{\left(
\omega t-\frac{l\pi}{2}\right)  }\sum_{j=0}^{\left[  \frac{2k+l+1}{2}\right]
}{C_{2k+l+1-2j}\left(  \omega t\right)  ^{2k+l+1-2j}},
\end{align*}
where $C_{2k+l+1}=-1,$ $C_{2k+l-j}=\left(  -1\right)  ^{j}\left[
P_{j+1}-\left(  2k+l+1-j\right)  C_{2k+l+1-j}\right]  ,$ $j=0,1,\ldots,2k+l,$
$P_{j}=\left(  -1\right)  ^{\left[  \frac{j}{2}\right]  }\frac{\left(
l+j\right)  !}{j!\left(  l-j\right)  !2^{j}}$ for $j=0,1,\ldots,l$ and
$P_{j}=0$ for $j>l$ .

In general one can use the formula \cite[(1.8.1.1)]{PrudBriMar} according to
which
\begin{align*}
\int_{0}^{x}{t^{2k+l+\frac{3}{2}}\mathbf{J}_{l+\frac{1}{2}}\left(  \omega
t\right)  }dt=\frac{\omega^{l+1/2} x^{2\left(  k+l\right)  +3}%
}{2^{l+\frac{1}{2}}\left[  2\left(  k+l\right)  +3\right]  \Gamma{\left(
l+\frac{3}{2}\right)  }}\hspace{0.1mm}_{1}\mathbf{F}_{2}\left(  \frac{2\left(
k+l\right)  +3}{2};\frac{2\left(  k+l\right)  +5}{2},l+\frac{3}{2}%
;-\frac{\left(  \omega x\right)  ^{2}}{4}\right)  .
\end{align*}

The derivative of (\ref{uN}) can also be written in a closed form
\begin{equation}\label{duN}
\begin{split}
\frac{d\mathbf{u}_{N}\left(  x\right)  }{dx}  &  =\frac{1}{2\sqrt{x}}\mathbf{J}%
_{l+\frac{1}{2}}\left(  \omega x\right)  +\frac{\omega\sqrt{x}}{2}\left\{
\mathbf{J}_{l-\frac{1}{2}}\left(  \omega x\right)  -\mathbf{J}_{l+\frac{3}{2}%
}\left(  \omega x\right)  \right\} \\
&  \quad+ \sum_{n=0}^{N}a_{n}\sum_{k=0}^{n} \Xi_{n,k}\left( u_{0}^{\prime
}\left(  x\right) \widetilde{X}^{\left(  2\left(  n-k\right)  \right)
}\left(  x\right)  - \frac{\widetilde{X} ^{\left(  2\left(  n-k\right)
-1\right)  }\left(  x\right) }{u_{0}(x)}\right)  \int_{0}^{x}t^{2k+l+1}%
\mathbf{d}_{l}\left(  t,\omega\right) \, dt\\
&  \quad+\mathbf{d}_{l}\left(  x,\omega\right)  u_{0}\sum_{n=0}%
^{N}{a_{n}\sum_{k=0}^{n}{\Xi_{n,k}{x^{2k+l+1}\widetilde{X}}^{\left(  2\left(
n-k\right)  \right)  }\left(  x\right)  }}.
\end{split}
\end{equation}

\section{Numerical solution of spectral problems}
One of the possible applications of the proposed approximation is the numerical solution of spectral problems. Recall (see, e.g., \cite{CasKravTor}, \cite{KosTesh2011}, \cite{Zettl}) that the classical formulation of spectral problems for equation \eqref{Intro Bessel perturbed} consists in finding the values of the spectral parameter for which the regular solution $u$ of \eqref{Intro Bessel perturbed} satisfies
\begin{equation}
\beta u(b,\lambda)+\gamma u^{\prime}(b,\lambda)=0 \label{EqBC2}%
\end{equation}
for some $\beta,\gamma\in\mathbb{C}$ such that $|\beta|+|\gamma|\neq0$. The regular solution (which is bounded around $x=0$) is unique up to a multiplicative constant for $l>0$, so no additional boundary condition is required, while for $-1/2\le l<0$ one of the equivalent ways to define the regular solution is by its asymptotics $u(x,\lambda)\sim x^{l+1}$, $x\to 0$. In both cases it is the regular solution considered in previous sections.

The following algorithm can be proposed for approximate solution of spectral problems for equation \eqref{Intro Bessel perturbed}.

\begin{enumerate}
\item Transform equation \eqref{Intro Bessel perturbed} into the form \eqref{Bessel perturbed transformed} by taking $\Lambda:=\lambda -q(0)$ and $q_0:=q-q(0)$.
\item Find a particular solution $u_0$ of \eqref{Bessel perturbed transformed} for $\Lambda=0$ satisfying the asymptotic conditions
    \[
    u_0(x)\sim x^{l+1}\qquad \text{and}\qquad u'_0(x)\sim (l+1)x^l,\quad x\to 0.
    \]
    The method described in \cite[Section 3]{CasKravTor} may be used.
\item Construct the functions $\{\mathbf{c}_n\}_{n\ge 0}$ using \eqref{cn def}.
\item Find $N$ and the coefficients $a_0,\ldots,a_N$ by minimizing the left hand side expression in \eqref{Q approx}, i.e., by solving one-dimensional uniform approximation problem.
\item Find approximate eigenvalues as zeros of \eqref{EqBC2}, the approximate solution $\mathbf{u}_N$ and the approximate derivative $\mathbf{u}'_N$ are calculated for all required values of $\lambda$ by \eqref{uN expanded} and \eqref{duN}.
\end{enumerate}
We refer the reader to \cite{CasKravTor}, \cite{KhmeKravTorTrem} and \cite{KT 2015 JCAM} for useful hints for realization of the proposed algorithm, in particular, for methods of calculation of the recursive integrals in \eqref{Xtilde} and of solution of the approximation problem \eqref{Q approx}.

We illustrate the proposed algorithm by the following numerical example.
\begin{Example}\label{Ex1}
Consider the following equation
\begin{equation}\label{Eq x2}
    -u''(x)+\frac{l(l+1)}{x^2}u(x) + x^2 u(x) = \lambda u(x),\qquad x\in (0,\pi],
\end{equation}
and consider two different spectral problems for equation \eqref{Eq x2}, one given by the Dirichlet boundary condition
\begin{equation}\label{BC Dirichlet}
u(\pi,\lambda)=0,
\end{equation}
and another one given by the Neumann boundary condition
\begin{equation}\label{BC Neumann}
u'(\pi,\lambda)=0.
\end{equation}

We implemented the proposed algorithm in Matlab 2012. All computations were performed in the double machine precision. All the integrals involved were calculated using the Newton-Cottes 6 point modified integration formula and all the functions involved were represented by their values on 20001 uniformly distributed points. The approximation problem \eqref{Q approx} was solved by the Remez algorithm, see \cite{KhmeKravTorTrem}. Exact characteristic functions for both spectral problems can be written in terms of the Kummer confluent hypergeometric functions, we used Wolfram Mathematica to obtain the ``exact'' eigenvalues from these characteristic functions.

\begin{table}[h]
\centering
\begin{tabular}{ccc}
\begin{tabular}{ccc}
$n$ & $\lambda_n$ & $\Delta \lambda_n$\\
\hline
1 &    2.001805251890 & $3.0\cdot 10^{-11}$\\
2 &    6.103446111613 & $4.1\cdot 10^{-11}$\\
3 &   10.956972252829 & $7.6\cdot 10^{-11}$\\
5 &   25.937702288545& $1.1\cdot 10^{-10}$\\
10 &  98.393217235689 & $5.5\cdot 10^{-10}$\\
20 & 393.38152313544 & $5.5\cdot 10^{-11}$\\
50 & 2478.3783027695 & $2.4\cdot 10^{-10}$\\
100 & 9953.3778488114 & $8.4\cdot 10^{-11}$ \\
\hline
 & $\max_{n\le 100} \Delta\lambda_n$ & $8.7\cdot 10^{-10}$\\
\hline
& $N$ in \eqref{Q approx} & 15\\
& $\varepsilon$ in \eqref{Q approx} & $1.4\cdot 10^{-10}$\\
\hline
\end{tabular}
&
\begin{tabular}{ccc}
$n$ & $\lambda_n$ & $\Delta \lambda_n$\\
\hline
1 &    4.015495482801 & $7.2\cdot 10^{-12}$\\
2 &    8.347478350031 & $2.5\cdot 10^{-11}$\\
3 &   13.920779935452 & $9.5\cdot 10^{-11}$\\
5 &   30.849480217032& $6.2\cdot 10^{-10}$\\
10 & 108.29669952501 & $3.2\cdot 10^{-10}$\\
20 & 413.28164218574 & $2.0\cdot 10^{-10}$\\
50 & 2528.2772347906 & $5.6\cdot 10^{-11}$\\
100 &10053.276592794 & $1.0\cdot 10^{-10}$ \\
\hline
 & $\max_{n\le 100} \Delta\lambda_n$ & $6.9\cdot 10^{-10}$\\
\hline
& $N$ in \eqref{Q approx} & 14\\
& $\varepsilon$ in \eqref{Q approx} & $1.4\cdot 10^{-10}$\\
\hline
\end{tabular}
&
\begin{tabular}{cc}
$n$ & $\Delta \lambda_n$\\
\hline
1 &    $7.9\cdot 10^{-6}$\\
2 &   $3.7\cdot 10^{-5}$\\
3 &    $2.2\cdot 10^{-4}$\\
5 &    $4.9\cdot 10^{-3}$\\
10 &  $2.6\cdot 10^{-3}$\\
20 &  $8.5\cdot 10^{-4}$\\
50 &  $9.0\cdot 10^{-4}$\\
100 & $8.3\cdot 10^{-6}$ \\
\hline
$\max\Delta\lambda_n$ & $4.9\cdot 10^{-3}$\\
\hline
 $N$ & 14\\
 $\varepsilon$ & $7.5\cdot 10^{-4}$\\
\hline
\end{tabular}\\
$l=-1/2$ & $l=1/2$ & $l=1$
\end{tabular}
\caption{The exact eigenvalues $\lambda_{n}$ for the spectral problem \eqref{Eq x2}, \eqref{BC Dirichlet} and the absolute errors $\Delta \lambda_n$ of the approximate eigenvalues calculated using the proposed algorithm for different values of $l$. The last two rows show the error achieved when solving the approximation problem \eqref{Q approx} and the parameter $N$ used for the construction of the approximate solutions.}
\label{Ex1Table1}
\end{table}

\begin{table}[h]
\centering
\begin{tabular}{ccc}
\begin{tabular}{ccc}
$n$ & $\lambda_n$ & $\Delta \lambda_n$\\
\hline
1 &    1.997907144139 & $2.1\cdot 10^{-11}$\\
2 &    5.868682400048 & $2.5\cdot 10^{-11}$\\
3 &    9.324466618141 & $7.9\cdot 10^{-11}$\\
5 &   21.551367196747& $3.8\cdot 10^{-10}$\\
10 &  88.908090757514 & $3.5\cdot 10^{-10}$\\
20 & 373.884502196653 & $1.1\cdot 10^{-10}$\\
50 & 2428.87873010418 & $2.4\cdot 10^{-10}$\\
100 & 9853.87795220908 & $8.0\cdot 10^{-11}$ \\
\hline
 & $\max_{n\le 100} \Delta\lambda_n$ & $5.7\cdot 10^{-10}$\\
\hline
\end{tabular}
&
\begin{tabular}{ccc}
$n$ & $\lambda_n$ & $\Delta \lambda_n$\\
\hline
1 &   3.980969145691 & $2.6\cdot 10^{-11}$\\
2 &   7.612337613278 & $1.3\cdot 10^{-10}$\\
3 &  11.447054658026 & $1.3\cdot 10^{-10}$\\
5 &  25.935713162142& $5.1\cdot 10^{-11}$\\
10 & 98.310440283471 & $7.1\cdot 10^{-10}$\\
20 &393.284522228857 & $1.3\cdot 10^{-10}$\\
50 &2478.27765466469 & $4.1\cdot 10^{-10}$\\
100 &  9754.77670125413 & $2.5\cdot 10^{-10}$ \\
\hline
 & $\max_{n\le 100} \Delta\lambda_n$ & $9.6\cdot 10^{-10}$\\
\hline
\end{tabular}
\\
$l=-1/2$ & $l=1/2$
\end{tabular}
\caption{The exact eigenvalues $\lambda_{n}$ for the spectral problem \eqref{Eq x2}, \eqref{BC Neumann} and the absolute errors $\Delta \lambda_n$ of the approximate eigenvalues calculated using the proposed algorithm for different values of $l$.}
\label{Ex1Table2}
\end{table}

We solved the spectral problem \eqref{Eq x2}, \eqref{BC Dirichlet} for three values of $l$, namely, for $l=-1/2$, $l=1/2$ and $l=1$, and the spectral problem \eqref{Eq x2}, \eqref{BC Neumann} for $l=-1/2$ and $l=1/2$ (for $l=1$ Wolfram Mathematica failed to provide us with a usable analytic expression for the characteristic function of the problem). The obtained eigenvalues and their absolute errors are reported in Tables \ref{Ex1Table1} and \ref{Ex1Table2}. We emphasize the remarkable accuracy of the eigenvalues for the first two values of $l$. Equations with these values of the parameter $l$ appear when one implements separation of variables in domains with cylindrical symmetry, e.g., in fiber optics \cite{RCasVKravSTor}, and present difficulties (especially the limit value $l=-1/2$) for various numerical packages. For example, the recent package \textsc{Matslice} 2.0 \cite{LedouxVDaele} requires to use $l=-0.499999988$ as a workaround in order to obtain correct approximate eigenvalues, while fails for $l=-1/2$. See also \cite[Example 7.5]{CasKravTor}. For larger values of $l$ the precision deteriorates as we illustrate for $l=1$, this is mainly due to limitation of machine precision arithmetics and large coefficients appearing during the solution of approximation problem \eqref{Q approx}, similar phenomena was observed in \cite{KT 2015 JCAM}. We leave detailed analysis of possible techniques to overcome this restriction for a separate study.
\end{Example}

\begin{Remark}
Similarly to the proof in \cite{Volk} one can see that the Goursat problem
\begin{equation*}%\label{Goursat modified}
    \frac{dK(x,x)}{dx}=h(x)
\end{equation*}
for equation \eqref{waveq} is well-posed meaning that if one approximate $q/2$ by the derivatives $\mathbf{c}_n'$, i.e., if one finds $N$ and coefficients $a_0,\ldots,a_N$ such that the inequality
\begin{equation}\label{q approx}
    \left|\frac{q(x)}{2}-\sum_{n=0}^N a_n\mathbf{c}_n'(x)\right|<\varepsilon
\end{equation}
holds for all  $x\in[0,b]$, then there exists such constant $C$, dependent on $q$, $l$ and $b$ only, that
\begin{equation}\label{K uniform}
    \bigl| K(x,t) - K_N(x,t)\bigr|<C\varepsilon,\qquad 0<t\le x\le b,
\end{equation}
where $K_N$ is given by \eqref{3.27}. The uniform approximation \eqref{K uniform} of the integral kernel $K$ provides uniform error bound for the approximate solution $\mathbf{u}_N$, see \cite{KT 2015 JCAM}, a stronger result compared to Theorem \ref{Thm Approx Sol}. The necessary condition for the approximation problem \eqref{q approx} to have a solution is that $q(0)=0$, which is not a problem, c.f., Section \ref{Sect 5}. Completeness of the functions $\{\mathbf{c}_n'\}_{n\ge 0}$ in the space $C_0([0,b])$ is a sufficient condition. Unfortunately we do not have a proof of this result. Nevertheless, we can propose a modification of the algorithm. Instead of minimizing \eqref{Q approx} in Step 4, one minimizes \eqref{q approx} in order to find $N$ and the coefficients $a_0,\ldots,a_N$. We checked the performance of this modified algorithm on the spectral problems from Example \ref{Ex1}. For $l=-1/2$ the error achieved in \eqref{q approx} was $7.3\cdot 10^{-13}$ for $N=19$ resulting in $\max_{n\le 100}\Delta\lambda_n\approx 3.6\cdot 10^{-12}$. For $l=1/2$ the error achieved in \eqref{q approx} was $7.6\cdot 10^{-13}$ for $N=22$ resulting in $\max_{n\le 100}\Delta\lambda_n\approx 3.6\cdot 10^{-12}$. However, for $l=1$ the approximation was worse, the error achieved in \eqref{q approx} was $6.3\cdot 10^{-3}$ for $N=13$ resulting in $\max_{n\le 100}\Delta\lambda_n\approx 2.2\cdot 10^{-3}$. We leave a detailed study of this modification for the future work.
\end{Remark}
\appendix

\section{Appendix}

\label{Appendix} The Riemann function for the boundary value problem
(\ref{3.22}), (\ref{3.23}), (\ref{3.24}) has the form \cite{Volk}
\begin{equation*}
v\left(  z,s;\xi,\eta\right)  =%
\begin{cases}
v_{1}\left(  z,s;\xi,\eta\right)  & \text{if }0\le s<\eta<z<\xi,\\
v_{2}\left(  z,s;\xi,\eta\right)  & \text{if }0\le s<z<\eta<\xi,
\end{cases}
%\label{RF}%
\end{equation*}
with
\[
v_{1}\left(  z,s;\xi,\eta\right)  =\left(  \eta-z\right)  ^{l}\left(
s-\xi\right)  ^{l}\left(  s-z\right)  ^{-2l}\hspace{0.1mm}_{2}\mathbf{F}%
_{1}\left(  -l,-l;1;\frac{\left(  z-\xi\right)  \left(  s-\eta\right)
}{\left(  z-\eta\right)  \left(  s-\xi\right)  }\right)
\]
and%
\[
{v_{2}\left(  z,s;\xi,\eta\right)  =\frac{-\sin{\left(  \pi l\right)  }%
\Gamma^{2}{\left(  1+l\right)  }}{\pi\Gamma{\left(  2+2l\right)  }}%
\frac{\left(  \eta-\xi\right)  ^{1+2l}\left(  z-s\right)  }{\left(
s-\xi\right) ^{l+1}\left(  \eta-z\right)  ^{ l+1}}\hspace{0.1mm}_{2}%
\mathbf{F}_{1}\left(  1+l,1+l;2+2l;\frac{\left(  z-s\right)  \left(  \eta
-\xi\right)  }{\left(  z-\eta\right)  \left(  s-\xi\right)  }\right)  .}%
\]
It is continuous with the exception of the line $z=\eta$.

The aim of this appendix is to provide the uniform with respect to $\xi$,
$\eta$ estimates for the functions $v_{1}$, $v_{2}$ and for the integral
\eqref{3.21}. The main reason for the necessity in such new estimates is that
those presented in \cite{Volk} contain an error. For example, in \cite{Volk}
it is stated that $u(\xi,\eta)=O[(\xi-\eta)^{1+l-\rho}]$, $0\le\rho<1$. One
can easily obtain from this estimate that the integral kernel $K$ of the
transmutation operator should satisfy the inequality $|K(x,t)|\le
C(xt)^{1-\rho}$ and in particular, for $x=t$, the inequality $|K(x,x)|\le
Cx^{2-2\rho}$ which is impossible for any $\rho<1/2$ due to the Goursat
condition $\frac{dK(x,x)}{dx}=\frac12 q(x)$.

\begin{Lemma}
The functions $v_{1}$ and $v_{2}$ satisfy the following inequalities
\begin{align}
|v_{1}(z,0;\xi,\eta)| & \le C_{1}\frac{|\eta- z|^{l} \xi^{l}}{z^{2l}}, & z &
\in\left( \frac{2\xi\eta}{\xi+\eta},\xi\right) ,\label{v_1_est1}\\
|v_{1}(z,0;\xi,\eta)| & \le C_{2}\frac{|\xi- z|^{l} \eta^{l}}{z^{2l}}\left(
\log\left( \frac{(\xi-z)\eta}{(z-\eta)\xi}\right) +C_{3}\right) , & z &
\in\left( \eta,\frac{2\xi\eta}{\xi+\eta}\right) ,\label{v_1_est2}\\
|v_{2}(z,0;\xi,\eta)| & \le C_{4}\frac{(\xi-\eta)^{1+2l}z}{\xi^{1+l}%
(\eta-z)^{1+l}}, & z & \in\left( 0,\frac{\xi\eta}{2\xi-\eta}\right)
,\label{v_2_est1}\\
|v_{2}(z,0;\xi,\eta)| & \le C_{5}\frac{(\xi-\eta)^{l}}{z^{l}}\left( \log\left(
\frac{z(\xi-\eta)}{\xi(\eta-z)}\right) +C_{6}\right) , & z & \in\left(
\frac{\xi\eta}{2\xi-\eta},\eta\right) ,\label{v_2_est2}%
\end{align}
where the constants $C_{i}$ do not depend on $z$, $\xi$ and $\eta$.
\end{Lemma}

\begin{proof}
Consider
\begin{equation}\label{v1 def F}
v_{1}\left(  z,0;\xi,\eta\right)  =\frac{\left(  -1\right)  ^{-l}\left(
\eta-z\right)  ^{l}\xi^{l}}{z^{2l}}\hspace{0.01cm}_{2}\mathbf{F}_{1}\left(
-l,-l;1;\sigma_1\right)
,\qquad 0<\eta<z<\xi,
\end{equation}
where  $\sigma_1:=\frac{(z-\xi)\eta}{(z-\eta)\xi}$. When the variable $z$ changes from $\eta$ to $\xi$, the variable $\sigma_1$ changes from $-\infty$ to $0$. In order to estimate the hypergeometric function in \eqref{v1 def F} we consider two cases, $\sigma_1\in (-1,0)$ and $\sigma_1\in(-\infty,-1)$, corresponding to $z\in \left(\frac{2\xi \eta}{\xi+\eta},\xi\right)$ and $z\in\left(\eta,\frac{2\xi \eta}{\xi+\eta},\right)$, respectively.
For the first case we have that $\hspace{0.01cm}_{2}\mathbf{F}_{1}\left(
-l,-l;1;\sigma_1\right)=\sum_{n=0}^\infty \left(\frac{(-l)_n}{n!}\right)^2 \sigma_1^n$, where $(\alpha)_n$ denotes the Pochhammer symbol. The terms of this series alternate in sign and their absolute values are monotone decreasing for $n>l$ and go to zero as $n\to\infty$. Hence,
\begin{equation}\label{1F2 first bound}
\left|\hspace{0.01cm}_{2}\mathbf{F}_{1}\left(
-l,-l;1;\sigma_1\right)\right| \le \sum_{n=0}^{[l+1]} \left(\frac{(-l)_n}{n!}\right)^2 |\sigma_1|^n\le \sum_{n=0}^{[l+1]} \left(\frac{(-l)_n}{n!}\right)^2 =: C_1,\qquad \sigma_1\in(-1,0).
\end{equation}
For the second case, assume first that $l\not\in\mathbb{N}$. We obtain using \cite[(2.1.4.18)]{Erdelyi} that
\begin{equation*}
\hspace{0.01cm}_{2}\mathbf{F}_{1}\left(  -l,-l;1;\sigma_{1}\right)
=\frac{\left(  -\sigma_{1}\right)  ^{l}}{\Gamma{\left(  -l\right)  }%
\Gamma{\left(  1+l\right)  }}\sum_{n=0}^{\infty}\left[  \frac{\left(
-l\right)  _{n}}{n!}\right]  ^{2}\sigma_{1}^{-n}\left[  \log\left(  -\sigma
_{1}\right)  +h_{n}\right],%\label{F}%
\end{equation*}
where $h_n = 2\psi(1+n) - 2\psi(n-l)+\pi \cot(\pi l)$. The series $\sum_{n=0}^\infty \left(\frac{(-l)_n}{n!}\right)^2 \sigma_1^{-n}$ is bounded by the same estimate \eqref{1F2 first bound} and, since $\psi(1+n) - \psi(n-l) = \frac{l+1}{n-l} +O(n^{-2})$, see \cite[(1.18.7)]{Erdelyi}, the series $\sum_{n=0}^\infty \left(\frac{(-l)_n}{n!}\right)^2 \sigma_1^{-n} \bigl(\psi(1+n) - \psi(n-l)\bigr)$ can be uniformly bounded by the absolutely convergent series $\sum_{n=0}^\infty \frac{l+1}{n-l}\left(\frac{(-l)_n}{n!}\right)^2$. Hence
\begin{equation}\label{1F2 second bound}
\begin{split}
\left|\hspace{0.01cm}_{2}\mathbf{F}_{1}\left(
-l,-l;1;\sigma_1\right)\right| & = \frac{|\sigma_1|^l}{|\Gamma(-l)|\Gamma(1+l)}\left|\bigl(\log(-\sigma_1) + \cot \pi l\bigr)\sum_{n=0}^\infty \left(\frac{(-l)_n}{n!}\right)^2 \sigma_1^{-n}\right.\\
&\quad \left. + \sum_{n=0}^\infty \left(\frac{(-l)_n}{n!}\right)^2\sigma_1^{-n}\bigl(2\psi(1+n) - 2\psi(n-l)\bigr)\right|\\
&\le \frac{|\sigma_1|^l}{|\Gamma(-l)|\Gamma(1+l)} \bigl( c_1 |\log(-\sigma_1) + \pi \cot\pi l| + c_2\bigr) \le C_2|\sigma_1|^l(\log(-\sigma_1) + C_3).
\end{split}
\end{equation}
In the case $l\in\mathbb{N}$ the hypergeometric function reduces to a polynomial, see \cite[(2.1.4.20)]{Erdelyi}, and the proof is straightforward.
Combining estimates \eqref{1F2 first bound} and \eqref{1F2 second bound} with \eqref{v1 def F} one easily obtains estimates \eqref{v_1_est1} and \eqref{v_1_est2}.
The proof of the estimates \eqref{v_2_est1} and \eqref{v_2_est2} is completely similar with the only difference that for the case $l\in \mathbb{N}$ one has to use the formula \cite[(2.1.4.19)]{Erdelyi}.
\end{proof}

In the next lemma we estimate the integrals involving the Riemann function.

\begin{Lemma}
\label{Estimates Riemann} The following inequalities hold
\begin{equation}
\int_{\eta}^{\xi}{\left\vert v_{1}\left(  z,0;\xi,\eta\right)  \right\vert
z^{l-\frac{1}{2}}}dz\leq A_{1}\left(  \xi-\eta\right)  ^{l} \left( \sqrt\xi-
\sqrt\eta\right)  \label{V1}%
\end{equation}
and
\begin{equation}
\int_{0}^{\eta}{\left\vert v_{2}\left(  z,0;\xi,\eta\right)  \right\vert
z^{l-\frac{1}{2}}}dz\leq 
\begin{cases}
    A_{2}\left(  \xi-\eta\right) ^{l} \left( \sqrt\xi-
    \sqrt\eta\right) , & \text {if } l> 0,\\
    A_{2}\left(  \xi-\eta\right) ^{l+1/2}\bigl(1-\frac \eta\xi\bigr)^{l+1/2}, & \text{if }-1/2\le l<0,
\end{cases}
    \label{V2}%
\end{equation}
where the constants $A_{1}$, $A_{2}$ do not depend on $\xi$ and $\eta$.
\end{Lemma}

\begin{proof}
We see using the estimates \eqref{v_1_est1} and \eqref{v_1_est2} that
\[
\begin{split}
\int_\eta^\xi \bigl|v_1(z,0;\xi,\eta)\bigr| z^{l-1/2}\,dz &\le C_2C_3 \eta^l \int_\eta^{\frac{2\xi\eta}{\xi+\eta}}\frac{(\xi-z)^l}{z^{l+1/2}}\,dz + C_1 \xi^l \int_{\frac{2\xi\eta}{\xi+\eta}}^\xi \frac{(z-\eta)^l}{z^{l+1/2}}\,dz
\\
& \quad+ C_2 \eta^l \int_\eta^{\frac{2\xi\eta}{\xi+\eta}}\frac{(\xi-z)^l}{z^{l+1/2}}\log \left(\frac{(\xi-z)\eta}{(z-\eta)\xi}\right) \,dz= :I_1+I_2+I_3.
\end{split}
\]
To bound the first integral note that $\frac{\xi-\eta}{2\eta}\le \frac{\xi-z}{z}\le \frac{\xi-\eta}{\eta}$, i.e., $\frac{(\xi-z)^l}{z^l}\le C_l\frac{(\xi-\eta)^l}{\eta^l}$, where $C_l=\max(1,2^{-l})$. Hence
\[
I_1\le C_2C_3C_l (\xi-\eta)^l \int_\eta^{\frac{2\xi\eta}{\xi+\eta}}\frac{dz}{\sqrt{z}}= 2C_2C_3C_l (\xi-\eta)^l\left(\sqrt{\frac{2\xi\eta}{\xi+\eta}} - \sqrt\eta\right) \le 2C_2C_3C_l (\xi-\eta)^l\left(\sqrt\xi - \sqrt\eta\right).
\]
The second integral can be estimated similarly. For the third integral note that the argument of the logarithmic function belongs to $[1,\infty)$ and for any $x\ge 1$ and $\varepsilon>0$ the following inequality holds
\[
\log x \le \frac{x^\varepsilon}{\varepsilon e}.
\]
We take $\varepsilon=1/2$ and similarly to the previous cases obtain that
\[
\begin{split}
I_3&\le \frac{2C_2\eta^{l+1/2}}{e \xi^{1/2}}\int _\eta^{\frac{2\xi\eta}{\xi+\eta}}\frac{(\xi-z)^{l+1/2}}{z^{l+1/2}(z-\eta)^{1/2}}\,dz\le \frac{2C_2(\xi-\eta)^{l+1/2}}{e \xi^{1/2}}\int _\eta^{\frac{2\xi\eta}{\xi+\eta}}\frac{dz}{(z-\eta)^{1/2}}\\
&=
\frac{4C_2(\xi-\eta)^{l+1/2}}{e \xi^{1/2}}
\frac{\eta^{1/2}(\xi-\eta)^{1/2}}{\sqrt{\xi+\eta}}\le \frac{8C_2}{e}(\xi-\eta)^{l}\left(\sqrt\xi - \sqrt\eta\right).
\end{split}
\]
Summing up the obtained inequalities we get \eqref{V1}.
For the integral \eqref{V2} using the estimates \eqref{v_2_est1} and \eqref{v_2_est2} we have
\begin{equation*}
\begin{split}
\int_{0}^{\eta}{\left\vert v_{2}\left(  z,0;\xi,\eta\right)  \right\vert
z^{l-\frac{1}{2}}}dz & \le C_4\frac{(\xi-\eta)^{1+2l}}{\xi^{1+l}} \int_0^{\frac{\xi \eta}{2\xi-\eta}} \frac{z^{l+1/2}}{(\eta-z)^{l+1}}\,dz +
C_5C_6(\xi-\eta)^l \int_{\frac{\xi \eta}{2\xi-\eta}}^\eta \frac{dz}{z^{1/2}}\\
& \quad+ C_5 (\xi-\eta)^l \int_{\frac{\xi \eta}{2\xi-\eta}}^\eta \log\left(\frac{z(\xi-\eta)}{\xi(\eta-z)}\right)\frac{dz}{z^{1/2}} =: I_4+I_5+I_6.
\end{split}
\end{equation*}
The integral $I_4$ can be bounded as follows. Let $l>0$. Then 
\begin{equation*}
\begin{split}
I_4 & = C_4\frac{(\xi-\eta)^{1+2l}}{\eta\xi^{1+l}} \int_0^{\frac{\xi \eta}{2\xi-\eta}} \left(\frac z{\eta-z}\right)^{l-1} z^{3/2} d\left(\frac z{\eta-z}\right) \\
& \le C_4\frac{(\xi-\eta)^{1+2l}}{\eta\xi^{1+l}} \frac{\eta^{3/2}\xi^{3/2}}{(2\xi-\eta)^{3/2}}\int_0^{\frac{\xi \eta}{2\xi-\eta}} \left(\frac z{\eta-z}\right)^{l-1} d\left(\frac z{\eta-z}\right)\\
&=C_4\frac{(\xi-\eta)^{l+1}\eta^{1/2}\xi^{1/2}}{l(2\xi-\eta)^{3/2}}\le \frac{2C_4}{l} (\xi-\eta)^l \left(\sqrt{\xi}-\sqrt{\eta}\right),
\end{split}
\end{equation*}
while for the case $-1/2\le l< 0$ we have
\begin{equation*}
    \begin{split}
       I_4 & \le C_4\frac{(\xi-\eta)^{1+2l}}{\xi^{1+l}} \frac{\eta^{l+1/2}\xi^{l+1/2}}{(2\xi-\eta)^{l+1/2}} \int_0^{\frac{\xi \eta}{2\xi-\eta}} \frac {dz}{(\eta-z)^{l+1}} \\
         & \le  C_4\frac{(\xi-\eta)^{1+2l}\eta^{l+1/2}}{\xi^{1+l}} \frac{\eta^{-l}}{(-l)}\le C_4\frac{(\xi-\eta)^{1+2l}}{(-l)\xi^{1/2+l}}.
     \end{split}
\end{equation*}
%and for $l=0$ one obtains
%\begin{equation*}
%    \begin{split}
%       I_4 & \le C_4\frac{(\xi-\eta)}{\xi} \frac{\eta^{1/2}\xi^{1/2}}{(2\xi-\eta)^{1/2}} \int_0^{\frac{\xi \eta}{2\xi-\eta}} \frac {dz}{\eta-z} \\
%         & =  C_4\frac{(\xi-\eta)\eta^{1/2}}{\xi^{1/2}}
%         \log\frac{2\xi-\eta}{\xi-\eta}\le C_4(\xi-\eta)\log\frac{\xi}{\xi-\eta}.
%     \end{split}
%\end{equation*}
For the integral $I_5$ we have
\begin{equation*}
\begin{split}
I_5 &= 2C_5C_6 (\xi-\eta)^l \left(\sqrt\eta - \sqrt\frac{\xi\eta}{2\xi-\eta}\right)
%2C_5C_6 (\xi-\eta)^l\frac{\eta - \frac{\xi\eta}{2\xi-\eta}}{\sqrt\eta + \sqrt\frac{\xi\eta}{2\xi-\eta}}\\
= 2C_5C_6 (\xi-\eta)^l \frac{\eta(\xi-\eta)}{(2\xi-\eta)\left(\sqrt\eta + \sqrt\frac{\xi\eta}{2\xi-\eta}\right)}\\
&\le 2C_5C_6 (\xi-\eta)^l \frac{\eta\left(\sqrt\xi+\sqrt\eta\right)\left(\sqrt\xi-\sqrt\eta\right)}{2\sqrt\eta(2\xi-\eta)}\le 2C_5C_6 (\xi-\eta)^l\left(\sqrt\xi-\sqrt\eta\right).
\end{split}
\end{equation*}
For the integral $I_6$ we change the logarithmic function as for the integral $I_3$ and obtain
\[
I_6\le \frac{2C_5(\xi-\eta)^{l+1/2}}{e \xi^{1/2}}\int_{\frac{\xi \eta}{2\xi-\eta}}^\eta \frac{dz}{(\eta-z)^{1/2}}=\frac{4C_5(\xi-\eta)^{l+1/2}}{e \xi^{1/2}}\frac{\eta^{1/2}(\xi-\eta)^{1/2}}{(2\xi-\eta)^{1/2}}\le \frac{8C_5}{e}(\xi-\eta)^l\left(\sqrt\xi-\sqrt\eta\right).
\]
Additionally for $l<0$ one has 
\[
\sqrt\xi -\sqrt\eta = \frac{\xi - \eta}{\sqrt\xi +\sqrt\eta} \le \frac{\xi - \eta}{\sqrt\xi}\le \frac{\xi - \eta}{\sqrt\xi}\cdot \left(\frac{\xi}{\xi-\eta}\right)^{-l} = \frac{(\xi-\eta)^{1+l}}{\xi^{l+1/2}},
\]
which finishes the proof.
\end{proof}

Note that $\sqrt\xi-\sqrt\eta\le\sqrt{\xi-\eta}$, hence we immediately obtain
the following corollary.

\begin{Corollary}
\label{Estimates Riemann 2} The following inequalities hold
\[
\int_{\eta}^{\xi}{\left\vert v_{1}\left(  z,0;\xi,\eta\right)  \right\vert
z^{l-\frac{1}{2}}}dz\leq A_{1}\left(  \xi-\eta\right)  ^{l+1/2}%
\]
and
\[
\int_{0}^{\eta}{\left\vert v_{2}\left(  z,0;\xi,\eta\right)  \right\vert
z^{l-\frac{1}{2}}}dz\leq A_{2}\left(  \xi-\eta\right) ^{l+1/2},
\]
where the constants $A_{1}$, $A_{2}$ do not depend on $\xi$ and $\eta$.
\end{Corollary}

\section*{Acknowledgements}
The authors are grateful to Markus Holzleitner for pointing us an inaccuracy in the proof of Lemma \ref{Estimates Riemann} in the first version of this paper.

\end{document}